\documentclass[12pt]{amsart}

\usepackage{graphicx}
\usepackage{amssymb,amsmath,amsthm,amscd}
\usepackage{mathrsfs}
\usepackage{enumerate}
\usepackage[usenames,dvipsnames]{color}
\usepackage[colorlinks=true, pdfstartview=FitV,
 linkcolor=blue,citecolor=blue,urlcolor=blue]{hyperref}
\usepackage[all]{xy}

\usepackage{verbatim}

\newtheorem{theorem}{Theorem}[section]
\newtheorem{proposition}[theorem]{Proposition}
\newtheorem{corollary}[theorem]{Corollary}
\newtheorem{conjecture}[theorem]{Conjecture}

\newtheorem{lemma}[theorem]{Lemma}

\theoremstyle{definition}
\newtheorem{definition}[theorem]{Definition}
\newtheorem{example}[theorem]{Example}
\newtheorem{remark}[theorem]{Remark}

\numberwithin{equation}{section}
\newcommand{\nc}{\newcommand}

\numberwithin{equation}{section}
\nc{\hs}{\hspace*}

\newcommand{\Z}{\mathbb{Z}}
\newcommand{\Q}{\mathbb{Q}}
\newcommand{\qQ}{\mathcal{Q}}
\newcommand{\C}{\mathbb{C}}

\newcommand{\g}{\mathfrak{g}}

\newcommand{\seteq}{\mathbin{:=}}
\newcommand{\soplus}{\mathop{\mbox{\normalsize$\bigoplus$}}\limits}
\newcommand{\isoto}[1][]{\mathop{\xrightarrow%
[{\raisebox{.3ex}[0ex][.3ex]{$\scriptstyle{#1}$}}]%
{{\raisebox{-.6ex}[0ex][-.6ex]{$\mspace{2mu}\sim\mspace{2mu}$}}}}}

\newcommand{\Hom}{\mathrm{Hom}}

\newcommand{\supp}{{\rm supp}}
\newcommand{\Ht}{{\rm ht}}

\newcommand{\Ker}{\mathrm{Ker}}
\newcommand{\cl}{\mathrm{cl}}
\newcommand{\aff}{\mathrm{aff}}

\newcommand{\Rep}{\mathrm{Rep}}
\newcommand{\Ca}{\mathscr{C}}
\newcommand{\K}{\mathrm{K}}

\newcommand{\rl}{\mathsf{Q}}   
\newcommand{\wl}{\mathsf{P}}   
\newcommand{\cm}{\mathsf{A}}  
\newcommand{\ko}{\mathbf{k}} 

\newcommand{\Se}{\mathscr{S}}

\newcommand{\Rnorm}{R^{\rm{norm}}}
\newcommand{\rmat}[1]{{\mathbf r}_{\mspace{-2mu}\raisebox{-.5ex}{${\scriptstyle{#1}}$}}}

\newcommand{\conv}{\mathbin{\mbox{\large $\circ$}}} 

\newcommand{\dconv}[2]{\overset{#2}{\underset{#1}{\conv}} }
\newcommand{\dtens}[2]{\overset{#2}{\underset{#1}{\tens}} }

\newcommand{\tens}{\mathop\otimes}

\newcommand{\Fun}{\mathcal{F}}
\newcommand{\QFun}[1]{\mathcal{F}^{(#1)}_Q}

\newcommand{\modt}{ \; \mathrm{mod}\; 2  }
\newcommand{\Lto}{\longrightarrow}
\newcommand{\Lgets}{\longleftarrow}
\newcommand{\redex}{\widetilde{w_0}}
\newcommand{\um}{\underline{m}}
\nc{\gl}{\mathfrak{gl}}
\nc{\col}{\colon}
\nc{\Ua}[1][\g]{U_{q}'(#1)}
\nc{\be}{\begin{enumerate}}
\nc{\ee}{\end{enumerate}}
\nc{\bnum}{\begin{enumerate}[{\rm(i)}]}
\nc{\bnam}{\begin{enumerate}[{\rm(a)}]}
\nc{\al}{\alpha}
\nc{\eq}{\begin{eqnarray}}
\nc{\eneq}{\end{eqnarray}}
\nc{\eqn}{\begin{eqnarray*}}
\nc{\eneqn}{\end{eqnarray*}}
\nc{\To}[1][{\hs{2ex}}]{\xrightarrow{\,#1\,}}
\nc{\shc}{\mathcal{C}}
\nc{\ba}{\begin{array}}
\nc{\ea}{\end{array}}
\nc{\vpi}{\varpi}
\nc{\U}[1][\g]{U'_q(#1)}
\nc{\bl}{\bigl(}
\nc{\br}{\bigr)}
\nc{\ro}{{\rm(}}
\nc{\rf}{{\rm)}}
\nc{\rev}{\mathrm{rev}}
\nc{\smod}{\mbox{-$\mathrm{mod}$}}
\nc{\vs}{\vspace*}
\nc{\cor}{\mathbf{k}}
\nc{\uo}{^{(1)}}
\nc{\ut}{^{(2)}}
\nc{\bc}{\begin{cases}}
\nc{\ec}{\end{cases}}
\nc{\la}{\lambda}
\nc{\gmod}{\mbox{-$\mathrm{gmod}$}}
\nc{\noi}{\noindent}
\nc{\Lem}{\begin{lemma}}
\nc{\enlemma}{\end{lemma}}
\nc{\epito}{\twoheadrightarrow}

\renewcommand{\Im}{\mathrm{Im}}
\nc{\nn}{\nonumber}
\nc{\eps}{\varepsilon}
\nc{\utt}{^{(t)}}

\begin{document}

\title[Symmetric Quiver Hecke algebras and R-matrices IV]
{Symmetric quiver Hecke algebras and R-matrices of \\ Quantum affine algebras IV}

\author[S.-J. Kang, M. Kashiwara, M. Kim, S.-j. Oh]{Seok-Jin Kang$^{1}$, Masaki Kashiwara$^{2}$,  Myungho Kim and Se-jin Oh$^{3}$}

\address{Department of Mathematical Sciences
         and
         Research Institute of Mathematics\\
         Seoul National University \\
        Gwanak-ro 1, Gwanak-gu, Seoul 151-747, Korea}

         \email{sjkang@snu.ac.kr}

\address{Research Institute for Mathematical Sciences\\
          Kyoto University\\ Kyoto 606-8502, Japan\\
          \& Department of Mathematical Sciences
         and
         Research Institute of Mathematics\\
         Seoul National University\\ Seoul 151-747, Korea}

         \email{masaki@kurims.kyoto-u.ac.jp}

\address{School of Mathematics, Korea Institute for Advanced Study\\ Seoul 130-722, Korea}
         \email{mhkim@kias.re.kr}

\address{Department of Mathematical Sciences, Seoul National University Gwanak-ro 1, Gwanak-gu, Seoul 151-747, Korea}
         \email{sejin092@gmail.com}

\thanks{$^{1}$This work was supported by NRF Grant  \# 2014-021261 and by NRF Grant \# 2010-0010753.}

\thanks{$^{2}$This work was partially supported by Grant-in-Aid for
Scientific Research (B) 22340005, Japan Society for the Promotion of
Science.}

\thanks{$^{3}$ This work was supported by BK21 PLUS SNU Mathematical Sciences Division.}

\date{February 25, 2015}

\begin{abstract}
Let $\U$ be a twisted affine quantum group of type $A_{N}^{(2)}$
or $D_{N}^{(2)}$ and let $\g_{0}$ be the finite-dimensional
simple Lie algebra of type $A_{N}$ or $D_{N}$.
For a Dynkin quiver of type $\g_{0}$, we define a full subcategory
${\mathcal C}_{Q}^{(2)}$ of the category of finite-dimensional
integrable $\U$-modules, a twisted version of the category
${\mathcal C}_{Q}$ introduced by Hernandez and Leclerc. Applying the
general scheme of affine Schur-Weyl duality, we construct an exact
faithful KLR-type duality functor ${\mathcal F}_{Q}^{(2)}\col
\text{Rep}(R) \rightarrow {\mathcal C}_{Q}^{(2)}$, where
$\text{Rep}(R)$ is the category of finite-dimensional modules
over the quiver Hecke algebra $R$ of type $\g_{0}$
with nilpotent actions of the generators $x_k$. We show that
${\mathcal F}_{Q}^{(2)}$ sends any simple object to a simple object
and induces a ring isomorphism $\K(\text{Rep}(R))
\isoto\K({\mathcal C}_{Q}^{(2)})$.
\end{abstract}

\maketitle

\section*{Introduction}

This paper is a continuation of our previous works on symmetric
quiver Hecke algebras and $R$-matrices of quantum affine algebras
\cite{KKK13A, KKK13B, KKKO14A}. In particular, this work can be
considered as a twisted version of \cite{KKK13B}.

The classical Schur-Weyl duality explains the close connection
between the representation theory of the symmetric group $\mathfrak{S}_k$
and the representation theory of general linear Lie
algebra $\gl_{n}$ \cite{S1, S2}.
In \cite{Jimbo86}, Jimbo introduced
a quantum version of Schur-Weyl duality between the Hecke algebra
$H_{k}$  and quantized universal enveloping algebra $U_{q}(\gl_{n})$.
In \cite{CP96B, Che, GRV94}, Chari-Pressley, Cherednik,
Ginzburg-Reshetikhin-Vasserot constructed a quantum affine
Schur-Weyl duality functor from the category of
finite-dimensional representations of the affine Hecke algebra
$H_{k}^{\text{aff}}$ to the category of finite-dimensional
integrable modules over the quantum affine algebra
$U_{q}(A_{N}^{(1)})$.

Using the symmetric quiver Hecke algebras and $R$-matrices of
quantum affine algebras, it was shown in \cite{KKK13A} that one can
construct quantum affine Schur-Weyl duality functors in
a more general setting.
Let $\{V_{s} \}_{s \in \mathcal{S}}$ be a family of good
modules over a quantum affine algebra $U_{q}'(\g)$. By investigating
the distribution of poles of normalized $R$-matrices
between the $V_s$'s,
we obtain a
symmetric quiver Hecke algebra $R$ and a $(U_{q}'(\g), R)$-bimodule
$\widehat{V}$. In this way, we construct the {\em KLR-type quantum
affine Schur-Weyl duality functor} (or {\em  KLR-type duality
functor} for brevity) $\mathcal{F}$ from the category
$\text{Rep}(R)$ to the
category ${\mathcal C}_{\g}$ of finite-dimensional integrable
$U_{q}'(\g)$-modules given by $M \mapsto \widehat{V} \otimes_{R} M$
$(M \in \text{Rep}(R))$.
Here, $\Rep(R)$ is the category of finite-dimensional modules over the quiver Hecke algebra $R$ on which the generators $x_k$ act nilpotently.
When $R$ is of finite $A$, $D$, $E$ type,
$\mathcal {F}$ is exact (\cite{Kim12}).
With various choices
of quantum affine algebras and good modules, one can construct a lot
of interesting and significant KLR-type duality functors.

 In \cite{KKK13A,KKKO14A}, by taking $\{V(\varpi_1)_{(-q)^{2s}}
\}_{s \in \Z}$ as the family of good modules over $U_q(A_N^{(t)})$
$(N \in \Z_{\ge 2}, \ t \in \{ 1,2\})$, we constructed the exact
functor $\mathcal{F}$ from $\Rep(R)$ to $\mathcal C^0_{A_N^{(t)}}$,
and investigated properties of $\mathcal{F}$ and the relationships
among the categories $\Rep(R)$, $\mathcal C^0_{A_N^{(1)}}$ and
$\mathcal C^0_{A_N^{(2)}}$. Here $\mathcal{C}_{\g}^0$ is the
category of finite-dimensional modules over the quantum affine
algebra $U_q'(\g)$ introduced in \cite{HL10} and $R$ is the quiver
Hecke algebra of type $A_{\infty}$. 

In \cite{KKK13B}, motivated by the work of Hernandez and Leclerc
\cite{HL11}, Kang, Kashiwara and Kim applied the general scheme to
the quantum affine algebras $U_{q}'(A_{N}^{(1)})$,
$U_{q}'(D_{N}^{(1)})$ and their fundamental representations.
Let $\g^{(1)}$ be an untwisted affine Kac-Moody algebra of type $A_{N}^{(1)}$
or $D_{N}^{(1)}$ and let $\g_{0}^{(1)}$ be the classical part of
$\g\uo$, the finite-dimensional simple
Lie algebra of type $A_{N}$ or $D_{N}$ inside $\g\uo$.
 We denote by $I^{(1)}$ and $I^{(1)}_0$
the index set of simple roots for $\g^{(1)}$ and $\g_0^{(1)}$, respectively.
Let $\Delta_{\g_{0}^{(1)}}^{+}$ be the set of positive roots of $\g_0^{(1)}$
and $\Pi_{\g_{0}^{(1)}}$ the set of simple roots.
For any Dynkin quiver $Q$ of type $\g_{0}^{(1)}$, we set $\Gamma_{Q}=
\phi^{-1}(\Delta_{\g_{0}^{(1)}}^{+} \times \{0\})\subset I^{(1)}_0\times\Z$, the Auslander-Reiten
quiver of $Q$.
(See Section 2 for more details including the
definition of $\phi$.)
We take the fundamental representations
$V(\varpi_{i})$ $(i \in I_{0}\uo)$ as the family of good modules over
$U_{q}'(\g^{(1)})$ and take the index set $J= \phi^{-1}(\Pi_{\g_{0}^{(1)}}
\times \{0\})$ together with the maps $s^{(1)}\col J \to I_{0}^{(1)}$ and $X^{(1)}\col J \to
{\mathbf k}^{\times}$ given by $s^{(1)}\col (i,p) \mapsto i$, $X^{(1)}\col (i,p)
\mapsto (-q)^p $ for $(i,p) \in J$. Then the corresponding
quiver Hecke algebra $R$ is of type $\g_{0}^{(1)}$ (i.e., $A_{N}$ or
$D_{N}$) and the KLR-type duality functor is an exact
functor ${\mathcal F}_{Q}^{(1)}\col \text{Rep}(R) \rightarrow {\mathcal
C}_{Q}^{(1)},$ where ${\mathcal C}_{Q}^{(1)}$ is the full
subcategory of ${\mathcal C}_{\g\uo}$ introduced in \cite{HL11}.
Moreover,
${\mathcal F}_{Q}^{(1)}$ sends any simple object to a simple object
and induces a ring isomorphism $\K(\text{Rep}(R))
\isoto\K({\mathcal C}_{Q}^{(1)})$.

In this paper, we extend these results on $\g\uo$
to the
twisted quantum affine algebras $U_{q}'(\g)$ of type $A_{N}^{(2)}$
($N\ge2$)
and $D_{N}^{(2)}$ ($N\ge 4$) (see Remark~\ref{rem:figure}).
As in the untwisted case, we take the fundamental
representations $V(\varpi_{i})$ $(i \in I_{0})$ as the family of
good modules. But to take care of twisted setting, we need to
develop some {\em twisted} techniques.

Let $\g^{(t)}=A^{(t)}_N$ ($t=1,2$) or
$\g^{(t)}=D^{(t)}_N$ ($t=1,2$). Let $\g_0^{(t)}$ be the classical part of $\g^{(t)}$.
Let $I^{(t)}$ and $I^{(t)}_0$ be the index set of simple roots for
$\g^{(t)}$ and $\g_0^{(t)}$, respectively.

For an arbitrary affine Kac-Moody algebra $\g$,
we denote by $\Se(\g)$  the quiver
with the evaluation modules of fundamental modules of $\Ua$ as vertices
and with the arrows determined by the poles of
the normalized R-matrices
(see \eqref{eq: Quiver S}).

Then we have a quiver isomorphism
\eq
\pi_{\g\uo}\col\Se_0(\g^{(1)}) \isoto \Se_0(\g^{(2)}).\label{eq:S01}
\eneq
where $\Se_{0}(\g^{(t)})$ is a connected component of $\Se(\g^{(t)})$
($t=1,2$) (Proposition~\ref{prop: quiver iso}).

Set $J = \phi^{-1}(\Pi_{\g_{0}^{(1)}} \times \{0\})$.
Then we have defined a map
$J\to I_0^{(1)}\times \cor^\times\simeq \Se(\g^{(1)})$.
Its image is contained in $\Se_0(\g^{(1)})$.
Composing it with \eqref{eq:S01}
we obtain
$s\ut\col J \to I_{0}^{(2)}$
and $X\ut\col  J \to {\mathbf k}^{\times}$.
Then the quiver $\Se^{J}$
associated with $J$ and $(X\ut, s\ut)$ (see \eqref{eq: def Se^J}) coincides
with $Q^{\text{rev}}$, the reversed quiver of $Q$.
Hence the associated quiver Hecke algebra
$R$ is of type $\g^{(1)}_0$, i.e. $R=R^{A_{N}}$ or $R^{D_{N}}$.
Thus, we
obtain an exact KLR-type duality functor
$${\mathcal F}_{Q}^{(2)}\col
\text{Rep}(R) \rightarrow {\mathcal C}_{\g^{(2)}}.$$
Note that we have already constructed an exact functor
${\mathcal F}_{Q}^{(1)}\col
\text{Rep}(R) \rightarrow {\mathcal C}_{\g^{(1)}}$
with the same quiver Hecke algebra $R$.

The main results of
this paper are the following:
\bnum
\item ${\mathcal F}_{Q}^{(2)}$ sends the simple objects to simple
objects.
\item ${\mathcal F}_{Q}^{(2)}$ gives a bijection between the set of
simple objects in $\text{Rep}(R)$ and the set of simple
objects in the full subcategory ${\mathcal C}_{Q}^{(2)}$ of $\shc_{\g^{(2)}}$ defined in
Definition \ref{def:CQ2}.
\item ${\mathcal F}_{Q}^{(2)}$ restricts to a faithful functor from
$\text{Rep}(R)$ to the category ${\mathcal C}_{Q}^{(2)}$.
\item ${\mathcal F}_{Q}^{(2)}$ induces a ring isomorphism
$\K(\text{Rep}(R)) \isoto \K({\mathcal C}_{Q}^{(2)})$.
\item
The ring isomorphism $\K({\mathcal
C}_{Q}^{(1)}) \isoto \K({\mathcal C}_{Q}^{(2)})$, which is
obtained by (iv) and the ring isomorphism
$\K(\text{Rep}(R)) \isoto \K({\mathcal C}_{Q}^{(1)})$,
preserves the dimension.
\end{enumerate}

 Note that a similar relationship between
$\mathcal{C}_{\g^{(1)}}$ and $\mathcal{C}_{\g^{(2)}}$ have been also
studied by Hernandez \cite{H01} with an approach using
$q$-character. He showed that there is a ring isomorphism between
the Grothendieck rings $K(\mathcal{C}_{\g^{(1)}})$ and
$K(\mathcal{C}_{\g^{(2)}})$, which sends  Kirillov-Reshetikhin
modules to  Kirillov-Reshetikhin modules. This is one of our
motivations. 

This paper is organized as follows. In Section 1, we briefly review
some of basic knowledge on quiver Hecke algebras and quantum affine
algebras. In Section 2, we recall the general construction of
KLR-type duality functors and its application to untwisted quantum
affine algebras. In Section 3, we investigate the relation between
untwisted and twisted quantum affine algebras. In Section 4, we
prove the main results of our paper. Section 5 deals with the
existence of non-zero $U_{q}'(\g)$-module homomorphisms between
evaluation modules of fundamental representations and their tensor products.

\vskip 3mm

\section{Quiver Hecke algebras and quantum affine algebras}

\subsection{Quantum groups associated with Cartan datum} In this subsection, we recall the definition of quantum groups.
Let $I$ be an index set. A Cartan datum $(\cm,\wl,\wl^\vee,\Pi,\Pi^\vee)$ consists of
\begin{enumerate}
\item a generalized Cartan matrix $\cm=(a_{ij})_{i,j \in I}$ satisfying
$$ a_{ii}=2 \ (i \in I), \quad a_{ij} \in \Z_{\le 0} \ (i \ne j) \quad \text{and} \quad a_{ij}=0 \iff a_{ji}=0,$$
\item a free abelian group $\wl$, called the {\em weight lattice},
\item $\wl^\vee= \Hom(\wl,\Z)$, called the {\em co-weight lattice},
\item $\Pi=\{ \alpha_i \ | \ i \in I\}$, called the {\em set of simple roots},
\item $\Pi^\vee=\{ h_i \ | \ i \in I\}$, called the {\em set of simple co-roots},
\end{enumerate}
such that
\begin{enumerate}
\item[{\rm (a)}] $\langle h_i,\alpha_j \rangle =a_{ij}$ for all $ij \in I$,
\item[{\rm (b)}] $\Pi$ and $\Pi^\vee$ are linearly independent sets,
\item[{\rm (c)}] for every $i \in I$, there exists $\Lambda_i \in \wl$ such that $\langle h_j,\Lambda_i\rangle =\delta_{i,j}$.
\end{enumerate}
We say that $\cm$ is {\em symmetrizable} if there exists a diagonal matrix
$\mathsf{D}={\rm diag}(\mathsf{s}_i  \in \Z_{>0} \ | \ i
\in I)$ such that $\mathsf{DA}$ is symmetric. We call $\rl\seteq \soplus_{i \in I} \Z
\alpha_i$ the {\em root lattice} and  $\rl^+\seteq \sum_{i \in
I} \Z_{\ge 0} \alpha_i$ the {\em positive root lattice}.
For $\beta = \sum_{i \in I} n_i \alpha_i \in \rl^+$, we set
$$\Ht(\beta) = \sum_{i \in I} n_i \in \Z_{\ge 0} \quad \text{and} \quad \supp(\beta) = \{ i \in I \ | \ n_i > 0 \}.$$

Let $\cm$ be a symmetrizable generalized Cartan matrix and let $q$ be an
indeterminate.
Let $( \ , \ )$ be a non-degenerate symmetric bilinear form on $\Q\otimes_\Z \wl$ satisfying
$\langle h_i, \lambda  \rangle = \dfrac{2(\alpha_i, \lambda)}{(\alpha_i, \alpha_i)}$ for every
$i \in I$ and $\lambda \in \Q\otimes_\Z \wl$.

 For $i \in I$ and $m\in \Z$, we set $q_i \seteq
q^{\frac{(\alpha_i, \alpha_i)}{2}}$, $[m]_i = (q_i^{m}-q_i^{-m})/(q_i-q_i^{-1})$ and $[m]_i! = \prod_{k=1}^m [k]_i$.
Let us denote by $d$ the smallest positive integer such that $d \frac{(\alpha_i, \alpha_i)}{2} \in \Z$ for all $i \in I$.

\begin{definition} The {\em quantum group} $U_q(\mathsf{g})$ associated with a Cartan datum $(\cm,\wl,\wl^\vee,\Pi,\Pi^\vee)$ is the associative $\Q(q^{1/d})$-algebra generated by
$e_i$, $f_i$ $(i \in I)$ and $q^h$ $(h \in d^{-1}  \wl^\vee)$ subject to the following relations:
\begin{enumerate}
\item[{\rm (i)}] $q^0=1, \ q^h q^{h'} = q^{h+h'}$ for $h,h' \in d^{-1}\wl^\vee$,
\item[{\rm (ii)}] $q^he_i = q^{\langle h,\alpha_i \rangle} e_iq^h$, $q^hf_i = q^{-\langle h,\alpha_i \rangle} f_iq^h$ for $h \in d^{-1}\wl^\vee$,
\item[{\rm (iii)}] $e_if_j-f_je_i = \dfrac{K_i-K_i^{-1}}{q_i-q_i^{-1}}$ for $K_i \seteq q^{ \frac{(\alpha_i, \alpha_i)}{2} h_i}$,
\item[{\rm (vi)}] $\displaystyle\sum_{k=0}^{1-a_{i,j}}(-1)^k e_i^{(1-a_{ij}-k)}e_je_i^{(k)}=\sum_{k=0}^{1-a_{i,j}}(-1)^kf_i^{(1-a_{ij}-k)}f_jf_i^{(k)}=0$ for $i \ne j$.
\end{enumerate}
Here $e_i^{(k)} \seteq e_i^k/[k]_i!$ and $f_i^{(k)} \seteq f_i^k/[k]_i!$ for $k \in \Z_{\ge 0}$.
\end{definition}

The quantum group $U_q(\mathsf{g})$ has a comultiplication $\Delta$ which is defined by
$$ \Delta(e_i)=e_i \tens K_i^{-1} + 1 \tens e_i, \ \ \Delta(f_i)=f_i \tens 1 + K_i \tens f_i \ \ \text{and} \ \ \Delta(q^h)=q^h \tens q^h.$$

\subsection{Quiver Hecke algebras}
In this subsection, we shall review symmetric quiver Hecke algebras.
Let $\ko$ be a field and $\cm=(a_{ij})_{i,j \in I}$ be a {\em symmetric}
Cartan matrix.
We denote by $\rl^+$ the positive root lattice associated with $\cm$.
We normalize the symmetric bilinear form $( \ , \ )$ by
$(\alpha_i,\alpha_j) =a_{ij}$.
We choose a matrix $\qQ^\cm=(\qQ_{ij})_{i,j \in I}$
\begin{equation}\label{eq: condition Q}
\begin{aligned}
\qQ_{i,j}(u,v) =
\bc c_{ij}(u-v)^{-a_{ij}}&\text{if $i\not=j$,}\\
 0&\text{if $i=j$.}\ec
\end{aligned}
\end{equation}
where $c_{ij}\in \ko^\times $ and we assume
$\qQ_{i,j}(u,v) = \qQ_{j,i}(v,u)$.

For $\beta \in \rl^+$ with $\Ht(\beta)=n$, we set
$$ I^\beta \seteq \{ \nu=(\nu_1,\ldots,\nu_n) \in I^n \ | \ \alpha_{\nu_1}+\ldots+\alpha_{\nu_n} = \beta \}. $$

The symmetric group $\mathfrak{S}_n = \langle s_i  \  | \ 1 \le i <n \rangle$ on $n$-letters acts on $I^n$ by place permutations in a natural way.

\begin{definition}[\cite{KL09,R08}]
For $\beta \in \rl^+$ with $\Ht(\beta)=n$, the {\em symmetric quiver Hecke
algebra} (or KLR-algebra) $R(\beta)$
 at $\beta$ associated with a matrix $\qQ^\cm$
is the $\ko$-algebra generated by three sets of generators
$$\{ x_k \ | \ 1 \le k \le n\}, \ \ \{ \tau_\ell \ | \ 1 \le \ell < n\}, \ \ \{ e(\nu) \ | \ \nu \in I^\beta\} $$
satisfying the following relations:
\begin{itemize}
\item[{\rm (i)}] $e(\nu) e(\nu') = \delta_{\nu,\nu'} e(\nu),\ \sum_{\nu \in I^{\beta}} e(\nu)=1$,
\item[{\rm (ii)}] $x_k e(\nu) =  e(\nu) x_k, \  x_k x_{k'} = x_{k'} x_k$,
\item[{\rm (iii)}] $\tau_\ell e(\nu) = e(s_\ell(\nu)) \tau_\ell,\  \tau_{\ell'} \tau_\ell = \tau_\ell \tau_{\ell'} \text{ if } |\ell - \ell'| > 1$,
\item[{\rm (iv)}]  $\tau_\ell^2 e(\nu) = \mathcal{Q}_{\nu_\ell, \nu_{\ell+1}}(x_\ell, x_{\ell+1}) e(\nu)$,
\item[{\rm (v)}]  $(\tau_\ell x_k - x_{s_\ell(k)} \tau_\ell ) e(\nu) =
\begin{cases}
-  e(\nu) & \hbox{if } k=\ell \text{ and } \nu_\ell = \nu_{\ell+1}, \\
e(\nu) & \hbox{if } k = \ell+1 \text{ and } \nu_\ell = \nu_{\ell+1},  \\
0 & \hbox{otherwise,}
\end{cases}$
\item[{\rm (vi)}] $( \tau_{\ell+1} \tau_{\ell} \tau_{\ell+1} - \tau_{\ell} \tau_{\ell+1} \tau_{\ell} )
e(\nu) =\delta_{\nu_{\ell},\nu_{\ell+2}}
\dfrac{\qQ_{\nu_\ell,\nu_{\ell+1}}(x_\ell,x_{\ell+1})-\qQ_{\nu_\ell,\nu_{\ell+1}}(x_{\ell+2},x_{\ell+1})}{x_{\ell}-x_{\ell+2}} e(\nu)$.
\end{itemize}
\end{definition}

Note that, for each $\alpha_i$, the quiver Hecke algebra $R(\alpha_i)$ is isomorphic to $\ko[x]$. Thus there exists an $R(\alpha_i)$-module $L(i)= \ko u(i)$
defined by
$x \cdot u(i) =0$.

We denote by $\Rep(R(\beta))$ the category of
finite-dimensional $R(\beta)$-modules $M$ such that
the actions of $x_k$ ($1\le k\le \Ht(\beta)$) on $M$ are nilpotent.

For $\beta,\gamma \in \rl^+$ with $\Ht(\beta)=m$ and $\Ht(\gamma)=n$, we define
the idempotent $e(\beta,\gamma)$ in $R(\beta+\gamma)$ by
$$e(\beta,\gamma)= \sum_{\nu \in I^\beta, \mu \in I^\gamma} e(\nu * \mu)$$
where $\nu * \mu=(\nu_1,\ldots,\nu_m,\mu_1,\ldots,\mu_n)$ for $\nu
\seteq(\nu_1,\ldots,\nu_m)\in I^\beta$ and $\mu\seteq(\mu_1,\ldots,\mu_n) \in I^\gamma$.

Let $\iota_{\beta,\gamma}$ be a $\ko$-algebra homomorphism
$$ \iota_{\beta,\gamma} \colon R(\beta) \tens R(\gamma) \Lto e(\beta,\gamma)R(\beta+\gamma)e(\beta,\gamma)  $$
given by
\begin{equation} \label{eq: iota beta gamma}
\begin{aligned}
& \iota_{\beta,\gamma}( e(\nu) \tens e(\mu) ) = e(\nu *\mu), \\
& \iota_{\beta,\gamma}(x_k \tens 1) = x_k e(\beta,\gamma) \ (1 \le k
\le m), \\
&  \iota_{\beta,\gamma}(1 \tens x_k) = x_{m+k} e(\beta,\gamma) \ (1 \le k \le n), \\
  & \iota_{\beta,\gamma}(\tau_k \tens 1) = \tau_k e(\beta,\gamma) \ (1 \le k < m), \\
 &\iota_{\beta,\gamma}(1 \tens \tau_k) = \tau_{m+k} e(\beta,\gamma) \ (1 \le k < n).
\end{aligned}
\end{equation}

For an $R(\beta)$-module $M$ and an $R(\gamma)$-module $N$, we define their {\em convolution product} $M \conv N$ as follows:
\begin{equation} \label{eq: convolution product}
M \conv N \seteq R(\beta+\gamma)e(\beta,\gamma) \tens_{R(\beta)\tens R(\gamma)} (M \tens N).
\end{equation}
Then the category $\Rep(R) \seteq \soplus_{\beta \in \rl^+}
\Rep(R(\beta))$ becomes a tensor category in the sense of \cite[Appendix A.1]{KKK13A} induced by the convolution product.
We denote by $\K(\Rep(R))$
the Grothendieck ring of the category $\Rep(R)$.

\subsection{Quantum affine algebras} In this subsection, we briefly recall the representation theory of the quantum affine algebra $U_q'(\g)$.
We refer to \cite{AK,CP94,Kas02} for the basics on the representation theory of quantum affine algebra $U_q'(\g)$.

\medskip

The affine Cartan datum $(\cm,\wl,\wl^\vee,\Pi, \Pi^{\vee})$ consists of
\begin{enumerate}
\item an affine Cartan matrix $\cm$,
\item a weight lattice $\wl = (\soplus_{i \in I} \Z \Lambda_i) \oplus \Z \delta$, where $\delta$ is the {\em imaginary root},
\item a set of simple roots $\Pi=\{ \alpha_i \ | \ i \in I\}$,  where
$$ \alpha_i = \begin{cases} \sum_{j \in I} a_{ji} \Lambda_j & \text{ if } i \in I_0,\\ \sum_{j \in I} a_{j0} \Lambda_j+\delta & \text{ if } i =0,\end{cases}$$
\item The simple coroots $h_i \in \wl^\vee$ are defined by
$\langle h_i, \Lambda_j\rangle=\delta_{i,j}$ and $\langle h_i, \delta \rangle =0$.
\end{enumerate}
Note that
$$ \Z \delta = \{ \lambda \in \rl \ | \ \langle h_i , \lambda \rangle =0 \text{ for all } i \in I \}.$$

Let us denote by $c = \sum_{i \in I} c_i h_i$ $(c_i \in \Z_{>0})$ the unique element in $\wl^\vee$ such that
$$ \Z c = \{ h \in \wl^\vee \ | \ \langle h ,\alpha_i \rangle =0 \text{ for all } i \in I \}.$$
Write $\delta = \sum_{i \in I} d_i \alpha_i$.
We take a symmetric bilinear form on $\Q \otimes_\Z \wl$ such that
$$(\alpha_i , \alpha_j) =c_i d_i^{-1} a_{ij} \quad (i,j \in I)
\quad\text{and $\langle h_i,\lambda\rangle=
\dfrac{2(\alpha_i,\lambda)}{(\alpha_i,\alpha_i)}$ for all $\lambda\in \wl$ and $i \in I$}.$$
Then we have $(\delta, \lambda) = \langle c, \lambda \rangle$ for all $\lambda\in \wl$.

We denote by
\begin{itemize}
\item $\g$ the affine Kac-Moody Lie algebra with an affine Cartan datum $(\cm,\wl,\wl^\vee,\Pi,\Pi^{\vee})$,
\item $U_q(\g)$ the quantum group associated with the affine Cartan datum $(\cm,\wl,\wl^\vee,\Pi,\Pi^{\vee})$,
\item $U_q'(\g)$ the subalgebra of $U_q(\g)$ generated by $e_i$, $f_i$ and $K_i^{\pm 1}$ for all $i \in I$.
\end{itemize}
We call $U_q'(\g)$ the {\em quantum affine algebra}.

Whenever we deal with the $U_q'(\g)$-modules,
we take the base field $\ko$ to be the algebraic closure of
$\C(q)$ in $\cup_{m >0} \C((q^{1/m}))$.

We take an element $0$ in  $I$
where $0$ denotes the leftmost vertices in the tables in \cite[pages 54, 55]{Kac} except $A^{(2)}_{2n}$-case in which we take the longest simple root
as $\alpha_0$. We set $I_0 \seteq I \setminus \{ 0\}$.

Let  $\g_0$ be
the finite-dimensional simple Lie subalgebra of $\g$ generated by $e_i$, $f_i$ for $i \in I_0$, called the {\em classical part} of $\g$,

Set $\wl_\cl \seteq \wl / \Z\delta$ and call it the {\em classical weight lattice}. Considering the canonical projection $\cl\col  \wl \to \wl_\cl$, we have
$$\wl_\cl = \soplus_{i \in I} \Z \cl(\Lambda_i).$$

We say that a $U_q'(\g)$-module $M$ is {\em integrable} if
\begin{itemize}
\item[{\rm (i)}] it is $\wl_\cl$-graded; i.e.,
$$M = \soplus_{\lambda \in \wl_\cl} M_\lambda \ \ \text{ where } M_\lambda=\{ u \in M \ | \ K_i u =q_i^{\langle h_i,\lambda \rangle} u \ \text{ for all } i \in I \},$$
\item[{\rm (ii)}] for all $i \in I$, $e_i$ and $f_i$ act on $M$ locally nilpotently.
\end{itemize}

We denote by $\Ca_\g$ the category of finite-dimensional integrable
$U_q'(\g)$-modules. Note that $\Ca_\g$ is a tensor category
 induced by the comultiplication $$\Delta|_{U_q'(\g)} \colon
U_q'(\g) \to U_q'(\g) \tens U_q'(\g).$$

A simple module $M$ in $\Ca_\g$ contains a non-zero vector $u$ of weight $\lambda\in \wl_\cl$ such that
\begin{itemize}
\item $\langle c,\lambda \rangle =0$ and $\langle h_i,\lambda \rangle \ge 0$ for all $i \in I_0$,
\item all the weight of $M$ are contained in $\lambda - \sum_{i \in I_0} \Z_{\ge 0} \cl(\alpha_i)$.
\end{itemize}
Such a $\la$ is unique and $u$ is unique up to a constant multiple.
We call $\lambda$ the {\em dominant extremal weight} of $M$ and $u$ the {\em dominant extremal weight vector} of $M$.

For a $\wl_\cl$-graded $U_q'(\g)$-module $M$, $M_\aff$ denotes the $\wl$-graded $U_q'(\g)$-module defined by
\begin{itemize}
\item[{\rm (i)}] $M_\aff=\soplus_{\lambda \in \wl}(M_\aff)_\lambda$ such that $(M_\aff)_\lambda \simeq (M)_{\cl(\lambda)}$,
\item[{\rm (ii)}] the actions of $e_i$ and $f_i$ are defined in such a way that they have weight $\pm\al_i$ and commute with the canonical projection
$\cl\col M_\aff \to M$.
\end{itemize}

For $x \in \ko^\times$ and $M \in \Ca_\g$, we define
\begin{align} \label{eq: aff}
M_x \seteq M_\aff / (z_M-x)M_\aff \end{align}
where $z_M$ denotes the $U_q'(\g)$-module automorphism of $M_\aff$ of weight $\delta$,
defined as the composition
$(M_\aff)_\lambda\isoto M_{\cl(\lambda)}\isoto (M_\aff)_{\lambda+\delta}$.
For each $i \in I_0$, we set $$\varpi_i \seteq {\rm
gcd}(c_0,c_i)^{-1}\cl(c_0\Lambda_i-c_i \Lambda_0) \in \wl_\cl.$$ Then there
exists a unique simple $U_q'(\g)$-module $V(\varpi_i)$ in $\Ca_\g$,
called the {\em fundamental module of weight $\varpi_i$}, satisfying
the following conditions:
\begin{itemize}
\item[{\rm (i)}] all the weights of $V(\varpi_i)$ are contained in the convex hull of $W_0 \varpi_i$,
\item[{\rm (ii)}] $\dim V(\varpi_i)_{\varpi_i}=1$ and $V(\varpi_i)= U_q'(\g) \cdot u_{\varpi_i}$,
\item[{\rm (iii)}] for any $\mu \in W_0 \varpi_i\subset\wl_\cl$, we can associate a non-zero vector $u_{\mu}$ of weight $\mu$ such that
$$ u_{s_j(\mu)} = \begin{cases} f_j^{(\langle h_j,\mu \rangle)} u_\mu & \text{ if } \langle h_j,\mu \rangle \ge 0 \\
e_j^{(- \langle h_j,\mu \rangle)}u_\mu & \text{ if } \langle h_j,\mu
\rangle \le 0\end{cases} \text{ for any $j \in I$}.$$
\end{itemize}
Here $W_0$ denotes the Weyl group
of $\g_0$ and $u_{\varpi_i}$ denotes the dominant extremal
weight vector of weight $\varpi_i$.

The module $V(\varpi_i)_x$ $(x\in \ko^\times)$ has the left dual
$\bl V(\varpi_i)_x\br^*$ and the right dual ${}^*\bl V(\varpi_i)_x\br$ with the following $U_q'(\g)$-module homomorphisms
(see \cite[\S 1.3]{AK}):
$$\bl V(\varpi_i)_x\br^* \tens V(\varpi_i)_x\To[\ {\mathrm{tr}}\ ] \ko \ \text{ and } \ V(\varpi_i)_x \tens {^*\bl V(\varpi_i)_x\br}\To[\ {\mathrm{tr}}\ ] \ko, $$
where
\begin{equation} \label{equation: psta}
\bl V(\varpi_i)_x\br^*  \simeq   V(\varpi_{i^*})_{x(p^*)^{-1}},
\ {}^*\bl V(\varpi_i)_x \br \simeq   V(\varpi_{i^*})_{xp^*} \ \  \text{ and } \ \
p^* \seteq (-1)^{\langle \rho^\vee ,\delta \rangle}q^{(\rho,\delta)}.
\end{equation}
Here $\rho$ (respectively,  $\rho^\vee$) denotes an element in $\wl$ (respectively,  $\wl^\vee$) such that
$\langle h_i,\rho \rangle=1$ (respectively,  $\langle \rho^\vee,\alpha_i  \rangle=1$) for all $i \in I$, and
$i\mapsto i^*$ denotes the involution on $I_0$ given by $\al_i=-w_0\,\al_{i^*}$,
where $w_0$ is the longest element of $W_0$.

We say that a $U_q'(\g)$-module $M$ is {\em good}
if it has a {\em bar involution}, a crystal basis with {\em simple
crystal graph}, and a {\em global basis} (see \cite{Kas02} for
precise definition). For instance, every fundamental module
$V(\varpi_i)$ for $i \in I_0$ is a good module. Note that
every good module is a simple $U_q'(\g)$-module. Moreover the tensor product of good modules is again good.
Hence any good module $M$ is {\it real} simple, i.e., $M\tens M$ is simple.

\vskip 3mm

\section{ KLR-type Schur-Weyl duality functors and Auslander-Reiten quiver}

In this section, we recall the {\em KLR-type
Schur-Weyl duality functor} $\mathcal{F}$ constructed in
\cite{KKK13A} which depends on the choice of a family of good
$U_q'(\g)$-modules and spectral parameters.  We also recall
the notion of Auslander-Reiten quiver $\Gamma_Q$ which was
used in \cite{KKK13B} to construct an exact functor
$\mathcal{F}^{(1)}_Q$ when $\g=A^{(1)}_{N}$ and $D^{(1)}_{N}$.

\subsection{KLR-type Schur-Weyl duality functors}
For simple $\U$-modules $M_1$ and $M_2$, there exists a unique $U_q'(\g)$-module
homomorphism
$$\Rnorm_{M_1,M_2}(z_1,z_2): (M_1)_\aff \tens (M_2)_\aff \to \ko(z_1,z_2)
\tens_{\ko[z^{\pm 1}_{1},z^{\pm 1}_{2}]} (M_{2})_\aff \tens (M_{1})_\aff $$
which sends $u_{1} \tens u_{s}$ to $u_{2} \tens u_{1}$.
Here $z_{k} \seteq z_{V_k}$ denotes the $U_q'(\g)$-automorphism of $(M_k)_\aff$ of weight $\delta$, and $u_k$ is a dominant extremal  weight 
vector of  $M_k$  ($k=1,2$).
We call the $U_q'(\g)$-homomorphism
$\Rnorm_{M_1,M_2}(z_{1},z_{2})$ the {\em normalized $R$-matrix} between $M_1$ and
$M_2$.

We denote by $d_{M_1,M_2}(z)$ the {\em denominator} of $\Rnorm_{M_1,M_2}(z_{1},z_{2})$ which is the monic
polynomial  of the smallest degree in $\ko[z]$ such that
\begin{align} \label{eq: denom}
d_{M_1,M_2}(z_{2}/z_{1})\Rnorm_{M_1,M_2}(z_{1},z_{2})\col  (M_1)_\aff \tens (M_2)_\aff
\to (M_2)_\aff \tens (M_1)_\aff.
\end{align}

Let $\{ V_s \}_{s \in \mathcal{S}}$ be a family of
good $U_q'(\g)$-modules labeled by an index set $\mathcal{S}$.
Let $u_s$ be a dominant extremal weight vector of $V_s$.

For a triple $(J,X,s)$ consisting of an index set $J$ and two maps $X\col J \to \ko^\times, \ s\col  J \to \mathcal{S}$,
we define a quiver $\Se^J=(\Se^J_0,\Se^J_1 )$ in the following way:
\begin{eqnarray}&&
\parbox{80ex}{
\begin{enumerate}
\item[{\rm (i)}] $\Se^J_0= J$.
\item[{\rm (ii)}] For $i,j \in J$, we put $\mathtt{d}_{ij}$ many arrows from $i$ to $j$, where $\mathtt{d}_{ij}$ is the order of the zero of
$d_{V_{s(i)},V_{s(j)}}(z)$ at $z=X(j)/X(i)$.
\end{enumerate}
} \label{eq: def Se^J}
\end{eqnarray}
Then $\Se^J$ is a quiver without loops and $2$-cycles, i.e.,
$\mathtt{d}_{ij}\mathtt{d}_{ji}=0$.
We define a symmetric Cartan matrix $\cm^J=(a_{ij})_{i,j \in
J}$ by
$$ a_{ij} = -\mathsf{d}_{ij}-\mathsf{d}_{ji} \text{ if } i \ne j, \text{ and } a_{ii}=2.$$

Let $\rl^+$ be the positive root lattice associated with
the Cartan matrix $\cm^J$. Then, for $\beta \in \rl^+$, we can
define the symmetric quiver Hecke algebra $R^J(\beta)$
associated with a matrix $\qQ^{\cm^J}=(\qQ_{ij})_{i,j \in J}$ such
that
$$ \qQ_{ij}(u,v)=(u-v)^{\mathsf{d}_{ij}}(v-u)^{\mathsf{d}_{ji}} \quad \text{ for all } i \ne j \in J.$$

\begin{theorem} [\cite{KKK13A}] \label{thm: QASWD functor}
For each $\beta \in \rl^+$, there exists a functor
$$ \Fun_\beta \colon \Rep(R^J(\beta)) \to \Ca_\g $$
enjoying the following properties:
\begin{enumerate}
\item[{\rm (a)}] $\Fun_{\alpha_i}(L(i)) \simeq (V_{s(i)})_{X(i)}$ for every $i \in J$,
 \  $\Fun_0(R^J(0)) \simeq \ko$  for every $n \in \Z$.

\item[{\rm (b)}] Set $$\Fun \seteq \soplus_{\beta \in \rl^+} \Fun_\beta \colon \Rep(R^J) \to  \Ca_\g.$$
Then for any $M \in \Rep(R^J(\beta))$ and $N \in \Rep(R^J(\gamma))$, there exists an
 $U'_q(\g)$-module  isomorphism $$ \Fun(M \conv N) \simeq  \Fun(M) \tens \Fun(N) $$
functorial in $M$ and $N$.
\item[{\rm (c)}] If the underlying graph of a quiver $\Se^J$ is of finite type $A$, $D$ or $E$, then the functor $\Fun$ is exact.
In particular, $\Fun$ induces a ring homomorphism
$$\K(\Rep(R^J)) \to \K(\Ca_{\g}).$$
\end{enumerate}
\end{theorem}

 We call  $\Fun$ the {\em   KLR-type Schur-Weyl duality
functor}.

\subsection{Auslander-Reiten quiver $\Gamma_Q$} In this subsection, we briefly review the combinatorial feature of the Auslander-Reiten quiver $\Gamma_Q$
of a Dynkin quiver $Q$ of a finite-dimensional simply-laced simple Lie algebra $\g_0$.

\medskip

Throughout this subsection, we denote by $w_0$ the longest element
of the Weyl group $W_0$ of $\g_0$, $\Pi_{\g_0}$ the set of simple roots of $\g_0$
and $\Delta^+_{\g_0}$ (respectively, $\Delta^-_{\g_0}$) the set of
positive (respectively, negative) roots of $\g_0$.

\medskip
Let
$Q=(Q_0,Q_1)$ be a Dynkin quiver of $\g_0$.
For a vertex $i$, we denote by $s_iQ$ the quiver obtained from $Q$ by
reversing all arrows in $Q_1$ incident to $i$.
We say that a vertex $i \in Q_0$ is a {\em source} (respectively,
{\em sink}) if every arrow incident to $i$ starts (respectively,  ends) at
$i$.
For a reduced expression $\widetilde{w}=s_{i_1} \cdots s_{i_l}$ of $w \in W_0$, we say that $\widetilde{w}$ is {\em adapted to} $Q$ if
$$\text{ $i_k$ is a source of the quiver $s_{i_{k-1}} \cdots s_{i_1}Q$  for all $1 \le k \le l$}.  $$
It is known that there exists a unique Coxeter element $\tau$
which has a reduced expression adapted to $Q$.
In fact, all the reduced expressions of such a $\tau$ are adapted to $Q$.

A map $\xi\col  Q_{0} \to \Z$ is called a {\em height function}
if $\xi_i=\xi_j+1$ if there exists an arrow $i \to j \in Q_1$. Since
$Q$ is connected, any pair $\xi$ and $\xi'$ of height functions differ by
a constant. We fix such a height function and
we define the repetition quiver $\widehat{Q}$ as follows:
\begin{itemize}
\item[{\rm (i)}] $\widehat{Q}_0 =\{ (i,p) \in Q_0 \times \Z \ | \ p -\xi_i \in 2\Z\}$,
\item[{\rm (ii)}] $\widehat{Q}_1 =\{(i,p) \to (j,p+1) \ | \
\text{$i$ and $j$ are adjacent in the Dynkin diagram of $\g_0$} \}$.
\end{itemize}
Note that there is a path from $(i,p)$ to $(j,r)$ in $\widehat Q$ if and only if $d(i,j) \le r-p$, where $d(i,j)$ is the distance of $i$ and $j$ in the
Dynkin diagram except the case $\g_0\simeq A_1$.

Set $\widehat{\Delta}_{\g_0} \seteq \Delta^+_{\g_0} \times \Z$. For $i \in I_0$, we define an element $\gamma_i \in \Delta^+_{\g_0}$ as follows:
\begin{equation} \label{eq: gamma_i}
\gamma_i = \sum_{j \in B(i)} \alpha_j,
\end{equation}
where $B(i)$ denotes the set of vertices $j$ such that there exists a path from $j$ to $i$ in $Q$.

We define a bijection $\phi\col  \widehat{Q}_0 \to \widehat{\Delta}_{\g_0}$ in an inductive way (\cite[\S 2.2]{HL11}):
\begin{eqnarray} &&
\parbox{80ex}{
\begin{enumerate}
\item[{\rm (i)}] $\phi(i,\xi_i)=(\gamma_i,0)$,
\item[{\rm (ii)}] for a given $\beta \in \Delta^+_{\g_0}$ with $\phi(i,p)=(\beta,m)$,
\begin{itemize}
\item if $\tau(\beta) \in \Delta^+_{\g_0}$, then $\phi(i,p-2) \seteq (\tau(\beta),m)$,
\item if $\tau(\beta) \in \Delta^-_{\g_0}$, then $\phi(i,p-2) \seteq (-\tau(\beta),m-1)$,
\item if $\tau^{-1}(\beta) \in \Delta^+_{\g_0}$, then $\phi(i,p+2) \seteq (\tau^{-1}(\beta),m)$,
\item if $\tau^{-1}(\beta) \in \Delta^-_{\g_0}$, then $\phi(i,p+2) \seteq (-\tau^{-1}(\beta),m+1)$.
\end{itemize}
\end{enumerate}
} \label{eq: phi}
\end{eqnarray}

The {\em Auslander-Reiten quiver} $\Gamma_Q$ is the full subquiver
of $\widehat{Q}$ whose set of vertices is $\phi^{-1}(\Delta^+_{\g_0}
\times  \{0\})$.

We say that a (partial) order $\preceq$ on $\Delta^+_{\g_0}$ is {\em convex}, if
it satisfies the condition:
\eq&&
\parbox{70ex}
{either $\beta \preceq \beta+\gamma \preceq \gamma$
or $\gamma \preceq \beta+\gamma \preceq \beta$
whenever $ \beta, \gamma \in \Delta^+_{\g_0}$ satisfy
$\beta+\gamma \in \Delta^+_{\g_0}$.}\eneq
The Auslander-Reiten quiver $\Gamma_Q$ defines a convex order $\preceq_Q$ on $\Delta^+_{\g_0}$ as follows (cf.\ \cite{R96}):
\begin{align} \label{eq: preceq Q}
\parbox{60ex}{$\beta \preceq_Q \gamma$ if and only if $d(i,j)\le a-b$,}
\end{align}
where $\phi^{-1}(\beta,0)=(i,a)$ and $\phi^{-1}(\gamma,0)=(j,b)$.
Note that (\cite{B99,R96}, see also \cite[\S 1]{Zelikson})  
\eq
&&\parbox{73ex}{$\beta\preceq_Q\gamma$ if and only if
there exists a path from $\phi^{-1}(\gamma,0)$ to $\phi^{-1}(\beta,0)$ in $\Gamma_Q$.}\label{eq:Qorder}
\eneq
On the other hand, for an arbitrary finite-dimensional simple Lie algebra $\g_0$,
any  reduced expression $\redex=s_{i_1}\cdots s_{i_r}$ of the longest element $w_0$
of $\g_0$ induces a convex total order $\le_{\redex}$ on
$\Delta^+_{\g_0}$
$$ \beta_k \le_{\redex} \beta_l \text{ if and only if } k\le l.$$
Here $\beta_k \seteq s_{i_1} \cdots s_{i_{k-1}} \alpha_{i_k}$ ($1 \le k \le r$).
Conversely, any convex total order is obtained in this way by
a reduced expression of $w_0$ (\cite{Papi}).

It is proved in \cite{B99} that,
for every reduced expression $\redex$ of $w_0$ adapted to $Q$,
the convex total order induced by $\redex$
is stronger than the order $\preceq_Q$,
i.e., for $\beta,\gamma \in \Delta^+_{\g_0}$,
$ \beta \preceq_{Q} \gamma$ implies $\beta  \le_{\redex} \gamma$.

For a convex order $\preceq$  and a positive root $\al$,
a pair of positive roots $(\beta,\gamma)$ is called a
{\em minimal pair of $\al$} if $\beta+\gamma=\al$,
$\beta\preceq\al\preceq\gamma$
 and
there exists no pair $(\beta',\gamma')$ other than $(\beta,\gamma)$
such that
$\beta'+\gamma'=\al$ and
$\beta\preceq\beta'\preceq\al\preceq\gamma'\preceq\gamma$.

\subsection{The functor $\QFun{1}$} In this subsection, we briefly review the properties of the functor $\QFun{1}$ which were investigated in \cite{KKK13B}.

\medskip

For simplicity, we denote by $d_{k,l}(z)$ or $d_{k,l}^{\;\g}(z)$ the denominator
$d_{V(\varpi_{k}),V(\varpi_l)}(z)$ between the fundamental
$U_q'(\g)$-modules $V(\varpi_k)$ and $V(\varpi_l)$ ($k,l \in I_0$)
(see \eqref{eq: denom}).

In \cite{AK,KKK13B}, the denominators $d_{k,l}(z)$ for $\g=A^{(1)}_{N}$ and $D^{(1)}_{N}$ are computed as follows. Note that we have $d_{k,l}(z)=d_{l,k}(z)$.
\begin{theorem} \label{thm: denominator of untwisted} \hfill
\begin{enumerate}
\item[{\rm (a)}] For $\g = A^{(1)}_{N}$ $(N \ge 2)$ and $1 \le k , l \le N$, we have
\begin{align} \label{eq: denominator A1}
d_{k,l}^{\;A^{(1)}_{N}}(z) = \prod_{s=1}^{\min(k,l,N+1-k,N+1-l)} (z-(-q)^{|k-l|+2s}).
\end{align}
\item[{\rm (b)}] For $\g = D^{(1)}_{N}$ $(N \ge 4)$ and $1 \le k , l \le N$, we have
\begin{equation} \label{eq: denominator D1}
\begin{aligned}
& d_{k,l}^{\; D^{(1)}_{N}}(z)= \begin{cases}
\displaystyle \prod_{s=1}^{\min (k,l)} \hspace{-1ex}(z-(-q)^{|k-l|+2s})(z-(-q)^{2N-2-k-l+2s}) & \text{ if } 1 \le k,l \le N-2, \allowdisplaybreaks \\
\ \displaystyle \prod_{s=1}^{k}(z-(-q)^{N-k-1+2s}) &  \hspace{-15ex}\text{ if }  1 \le k \le N-2 \text{ and } l \in \{ N-1, N\}, \allowdisplaybreaks \\
\ \displaystyle \prod_{s=1}^{\lfloor \frac{N-1}{2} \rfloor} (z-(-q)^{4s}) &  \text{ if }   k \neq l \in \{N,N-1\},  \allowdisplaybreaks  \\
\ \displaystyle \prod_{s=1}^{\lfloor \frac{N}{2} \rfloor} (z-(-q)^{4s-2}) &  \text{ if }  k=l \in \{ N-1, N\}.
 \end{cases}
\end{aligned}
\end{equation}
\end{enumerate}
\end{theorem}

The following theorem is one of the main results in \cite{KKK13B}.

\begin{theorem} [{\cite[Theorem 4.3.1]{KKK13B}}] \label{thm: simply laced Se^J }
Let $\g\uo$ be an affine Kac-Moody algebra of type $A^{(1)}_{N}$ or
$D^{(1)}_{N}$.
We take the index set
$$J \seteq \phi^{-1}(\Pi_{\g_0\uo} \times \{0 \})$$
and two maps $s\uo\col  J \to I\uo_0  \seteq \{ 1,2,\ldots,N \}$ and $X\uo\col J \to \ko^{\times}$ such that
$$s\uo(i,p)=i \quad \text{and} \quad X\uo(i,p)=(-q)^{p}.$$
Then the quiver $\Se^J$ in \eqref{eq: def Se^J}
coincides with the quiver $Q^{\rm{rev}}$ obtained from $Q$ by reversing all the arrows.
\end{theorem}

For a given Dynkin quiver $Q$ of $\g\uo_0$,
we set
$$V_Q(\beta)\seteq V(\varpi_{i})_{(-q)^p} \quad \text{where} \ \phi^{-1}(\beta,0)=(i,p).$$

\begin{definition}[\cite{HL11}]
The category $\Ca^{(1)}_Q$ is defined to be the smallest full
subcategory of $\Ca_{\g\uo}$ such that
\bnam
\item it contains $V_Q(\beta)$ for all $\beta \in \Delta^+_{\g_0}$ and the trivial representation,
\item it is stable under taking tensor products, subquotients and
extensions.
\ee
\end{definition}

We define an order $\prec_\Z$ on $\Z^r_{\ge 0}$ $(r \in \Z_{>0})$ as follows:
\begin{eqnarray}&&
\parbox{75ex}{
$\um'=(m'_1,\ldots,m'_r) \prec_\Z  \um=(m_1,\ldots,m_r) $ if and only if there exist integers $k$, $s$ such that $1 \le k \le s \le r$,
$m_t'=m_t$ $(t<k)$, $m'_k  <  m_k$, and $m_t'=m_t$ $(s<t\le r)$,
$m'_{s}  <  m_{s}$.
} \label{eq: prec Z}
\end{eqnarray}

\begin{theorem}[{\cite{Mc12}, see also \cite{Kato}}]   \label{thm: PBW}
Let $\g_0$ be a  finite-dimensional simple Lie algebra. We fix a reduced expression $\redex=s_{i_1}\cdots s_{i_r}$ of the longest element
$w_0 \in W_0$, and set $\beta_k=s_{i_1}\cdots s_{i_{k-1}}\al_{i_k}$.
Let $R$ be the quiver Hecke algebra corresponding to
$\g_0$. Then for every $\beta_k \in \Delta^+_{\g_0}$, there exists a
simple $R(\beta_k)$-module $S_{\redex}(\beta_k)$ satisfying the
following properties:
\bnum
\item For every $\um=(m_1,\ldots,m_r) \in \Z_{\ge 0}^r$, there exists a non-zero $R$-module homomorphism
\eqn&&\rmat{\um} \colon
\overset{\Lto}{\nabla}_{\redex}(\um) \seteq
S_{\redex}(\beta_1)^{\conv m_1}\conv\cdots \conv
S_{\redex}(\beta_r)^{\conv m_r}\\
&&\hs{20ex}{\To}\overset{\Lgets}{\nabla}_{\redex}(\um) \seteq
S_{\redex}(\beta_r)^{\conv m_r}\conv\cdots \conv
S_{\redex}(\beta_1)^{\conv m_1}.
\eneqn
Such a morphism is unique up to a constant multiple, and
${\rm Im}(\rmat{\um}) $ is a simple module and
coincides with the head of $\overset{\Lto}{\nabla}_{\redex}(\um)$
and the socle of $\overset{\Lgets}{\nabla}_{\redex}(\um)$.
\item For every $\um=(m_1,\ldots,m_r) \in \Z_{\ge 0}^r$, we have
$$[\overset{\Lto}{\nabla}_{\redex}(\um)] \in [{\rm Im}(\rmat{\um})] + \sum_{\um' \prec_\Z \um} \Z_{\ge 0} [{\rm Im}(\rmat{\um'})].$$
\item For any simple $R(\beta)$-module $M$, there exists
a unique $\um=(m_1,\ldots,m_r) \in \Z_{\ge 0}^r$ such that
$$M \simeq {\rm Im}(\rmat{\um}) \simeq {\rm hd}\big( \overset{\Lto}{\nabla}_{\redex}(\um) \big).$$
\item For any minimal pair $(\beta_k,\beta_l)$ of $\beta_j=\beta_k+\beta_l$,
there is an exact sequence
$$ \qquad 0 \to S_{\redex}(\beta_j)
\to S_{\redex}(\beta_k) \conv S_{\redex}(\beta_l) \overset{{\mathbf r}}{\Lto} S_{\redex}(\beta_l) \conv S_{\redex}(\beta_k)
\to S_{\redex}(\beta_j) \to 0$$
with a non-zero $\mathbf{r}$.
\end{enumerate}
\end{theorem}

\begin{remark} \label{rem: SQ(beta)}
In \cite{Mc12} (see also  \cite{Kato}), it is shown that  the
simple module $S_{\redex}(\beta)$ corresponds to the {\em dual PBW
generator} of $U^{-}_q(\g_0)$ of weight $-\beta$, under the
isomorphism between the  Grothendieck ring $\K(R\gmod)$ and the dual
integral form of $U^{-}_q(\g_0)$. Here $R\gmod$ is the category of
finite-dimensional {\it graded} $R$-modules.
 As a result, we have the following properties:
\begin{enumerate}
\item[{\rm (a)}] For every simple root $\alpha_i$, we have $S_{\redex}(\alpha_i) \simeq L(i)$.
\item[{\rm (b)}]  For any pair of reduced expressions $\redex$ and
$\widetilde{w_0}'$ of $w_0$ which are adapted to $Q$, we have
$$S_{\redex}(\beta) \simeq S_{\widetilde{w_0}'}(\beta) \quad \text{ for all } \beta \in \Delta^+_{\g_0}.$$
Thus we will denote by $S_Q(\beta)$ the simple $R(\beta)$-module $S_{\redex}(\beta)$ for a reduced expression $\redex$ adapted to $Q$.
\end{enumerate}
\end{remark}

From now on, we assume that $\g\uo=A^{(1)}_{N}$ or $D^{(1)}_{N}$, and take  a Dynkin quiver $Q$ of type $\g_0\uo$. We also denote by $R(\beta)$
the quiver Hecke algebra $R^J(\beta)$ associated with the triple
$(J,X\uo,s\uo)$ in Theorem \ref{thm: simply laced  Se^J }.
Hence the quiver Hecke algebra is of type $A_N$ or $D_N$.
Then we have an exact functor
$$ \Fun_Q^{(1)}: \Rep(R) \to \Ca_{\g\uo}.$$

By Theorem \ref{thm: QASWD functor} (a), one can observe that
$$\Fun_Q^{(1)}(L(j)) \simeq V_Q(\alpha_j) \quad \text{for any $j \in I_0$}.$$

\begin{theorem} [\cite{KKK13B}] \label{thm: QASWD functor Q} \hfill
\bnum
\item For every $\beta \in \Delta^+_{\g_0}$, we have
$$\Fun_Q^{(1)}(S_Q(\beta)) \simeq V_Q(\beta).$$
\item The functor $\Fun_Q^{(1)}$ sends  a simple object to a simple object and it induces a ring isomorphism
$$ \varphi^{(1)} \colon \K(\Rep(R))\isoto \K(\Ca^{(1)}_Q).$$
\end{enumerate}
\end{theorem}

\section{Similarities between untwisted cases and twisted cases} \label{sec: similar}

In this section, we first recall the denominators of normalized $R$-matrices of type
$A_{N}^{(2)}$ and $D_{N}^{(2)}$. Next, we shall observe similarities
between the representations of quantum affine algebras of untwisted types and
twisted types, i.e.,
the similarity between $A^{(1)}_N$ and $A^{(2)}_N$ ($N\ge2$)
and the similarity between $D^{(1)}_N$ and $D^{(2)}_N$ ($N\ge4$).

\begin{remark}\label{rem:figure}
In this paper, we follow  the enumeration  of vertices of Dynkin diagram of type $A^{(t)}_N$ ($t=1,2$, $N\ge2$)  and $D^{(t)}_N$ ($t=1,2$, $N\ge4$) as follows.
\begin{itemize}
\item $A^{(1)}_{N}$ ($N\ge 2$)\hs{5ex}
$
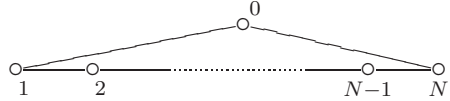
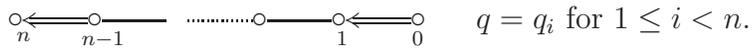
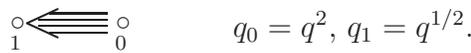
\xymatrix@R=2ex@C=4ex{ &&& *{\circ}<3pt> \ar@{-}[drrr]^<{0}\ar@{-}[dlll]\\
*{\circ}<3pt> \ar@{-}[r]_<{1}  &*{\circ}<3pt> \ar@{-}[r]_<{2}  & *{}<0pt>
\ar@{.}[r]&*{}<0pt>\ar@{.}[r]&*{}<0pt>\ar@{-}[r]_>{\,\,\,\,N-1}&*{\circ}<3pt>
\ar@{-}[r]_>{\,\,\,\,N}
&*{\circ}<3pt>}
$
\item $A_{2n}^{(2)}$ ($n\ge2$) \hs{5ex}
$ \xymatrix@R=3ex{ *{\circ}<3pt> \ar@{<=}[r]_<{n}  &*{\circ}<3pt>
\ar@{-}[r]_<{n-1}   & {} \ar@{.}[r] & *{\circ}<3pt>
\ar@{-}[r]_>{\,\,\,\ 1} &*{\circ}<3pt>\ar@{<=}[r]_>{\,\,\,\,0}
&*{\circ}<3pt> } $ \quad $q=q_i$ for $1\le i<n$.

\vs{1ex}
\item $A_{2}^{(2)}$ \hs{14ex}
$ \put(0,0){\circle{4}} \put(40,0){\circle{4}}
\put(12.1,-3.6){\line(1,0){22}} \put(8.1,-1.15){\line(1,0){26}}
\put(8,1.15){\line(1,0){26}} \put(12.1,3.6){\line(5,0){22}}
\put(1.5,-5.2){{\LARGE\mbox{$<$}}} \put(-3,-10){\tiny $1$}
\put(37,-10){\tiny $0$}
$\hs{15ex} \raisebox{-1ex}{$q_0=q^{2}$, $q_1=q^{1/2}$.}

\vs{1ex}
\item $A_{2n-1}^{(2)}$ ($n\ge 3$) \qquad
$
\xymatrix@R=3ex{ & *{\circ}<3pt> \ar@{-}[d]^<{0}\\
*{\circ}<3pt> \ar@{-}[r]_<{1}  &*{\circ}<3pt> \ar@{-}[r]_<{2}  & {}
\ar@{.}[r] & *{\circ}<3pt> \ar@{-}[r]_>{\,\,\,\ n-1}
&*{\circ}<3pt>\ar@{<=}[r]_>{\,\,\,\,n} &*{\circ}<3pt> }
$\quad \raisebox{-2ex}{$q=q_i$ for $i\not= n$.}

\item $A^{(2)}_3$\hs{15ex}
$\xymatrix@R=3ex{*{\circ}<3pt>  \ar@{<=}[r]_<{0} & *{\circ}<3pt> \ar@{=>}[r]_<{2} \ar@{}[r]_>{\,\,\,\ 1}  & *{\circ}<3pt>}$\quad $q=q_0=q_1$, $q_2=q^2$.

\item $D_{N}^{(1)}$ ($N\ge 4$) $ \qquad$
\xymatrix@R=1.2ex{
*{\circ}<3pt> \ar@{-}[dr]^<{0}
&&&&&&&*{\circ}<3pt>
\ar@{-}[dl]_<{N-1}\\
&*{\circ}<3pt> \ar@{-}[r]_<{2} & *{\circ}<3pt> \ar@{-}[r]_<{3}
& {} \ar@{.}[r]&{} \ar@{-}[r]_>{\,\,\,\,N-3}
& *{\circ}<3pt> \ar@{-}[r]_>{N-2\,\,\,} &*{\circ}<3pt>\\
*{\circ}<3pt> \ar@{-}[ur]_<{1}&&&&&&&
*{\circ}<3pt> \ar@{-}[ul]^<{N}
}

\item $D_{N}^{(2)}$ ($N\ge4$) \hs{3ex}
$
\xymatrix@R=2ex{ *{\circ}<3pt> \ar@{<=}[r]_<{0}  &*{\circ}<3pt>
\ar@{-}[r]_<{1}   & {} \ar@{.}[r] & *{\circ}<3pt>
\ar@{-}[r]_>{\,\,\,\ N-2} &*{\circ}<3pt>\ar@{=>}[r]_>{\,\,\,\,N-1}
&*{\circ}<3pt> }
$\quad $q=q_0=q_{N-1}$.
\end{itemize}

In particular, the affine Kac-Moody algebra of type $A^{(2)}_{3}$
should be understood as $D^{(2)}_3$.

In these cases, $i^*$ and $p^*$ in \eqref{equation: psta} are given as follows:
\begin{align}i^*&=\bc
N+1-i&\text{if $\g=A^{(1)}_N$,}\\
N-(1-\eps)&\text{if $\g=D_{N}\uo$, $N$ is odd and $i=N-\eps$, $\eps=0,1$,}\\
i&\text{otherwise,}\label{eq:dual1}
\ec\\
\intertext{and}
p^*&=\bc
(-q)^{N+1}&\text{if $\g=A^{(1)}_N$,}\\
(-q)^{2N-2}&\text{if $\g=D_{N}\uo$,}\\
-q^{N+1}&\text{if $\g=A^{(2)}_N$,}\\
(-1)^Nq^{2(N-1)}&\text{if $\g=D_{N}\ut$.}
\ec\label{eq:dual2}
\end{align}
\end{remark}

In \cite{Oh14}, the denominators of normalized $R$-matrices between fundamental modules of type
$A_{N}^{(2)}$ and $D_{N}^{(2)}$ are computed as in the following theorem.
Note that we have $d_{k,l}(z)=d_{l,k}(z)$.

\begin{theorem} \label{thm: denominator of twisted} \hfill
\begin{enumerate}
\item[{\rm (a)}]  For $\g = A^{(2)}_{N}$ $(N \ge 2)$
and $1 \le k , l \le \big\lfloor \dfrac{N+1}{2} \big\rfloor$, we have
\begin{align}
d_{k,l}^{A^{(2)}_{N}}(z) = \prod_{s=1}^{\min(k,l)} (z-(-q)^{|k-l|+2s})(z+q^{N+1}(-q)^{-k-l+2s}).
\end{align}
\item[{\rm (b)}]  For $\g = D^{(2)}_{N}$ $(N \ge 4)$ and $1 \le k , l \le N-1$, we have
\begin{equation} \label{eq: denominator 1}
\begin{aligned}
& d_{k,l}^{\;D^{(2)}_{N}}(z)= \begin{cases}
\displaystyle \prod_{s=1}^{\min(k,l)}\hs{-1ex} \big(z^2 - (-q^2)^{|k-l|+2s}\big)\big(z^2 - (-q^2)^{2N-2-k-l+2s}\big) & \text{ if } 1 \le k,l \le N-2, \allowdisplaybreaks \\
\ \displaystyle \prod_{s=1}^{k}\big(z^2+(-q^{2})^{N-1-k+2s}\big) &  \hspace{-14ex}\text{ if }  1 \le k < N-1 \text{ and } l =N-1, \allowdisplaybreaks \\
\ \displaystyle  \prod_{s=1}^{N-1} \big(z+(-q^2)^{s}\big) &  \text{ if }  k=l=N-1. \allowdisplaybreaks
 \end{cases}
\end{aligned}
\end{equation}
\end{enumerate}
\end{theorem}

\subsection{The Quiver $\Se_0(\g)$} For each $\g$, we define a quiver $\Se(\g)$ as follows:
\begin{eqnarray} &&
\parbox{80ex}
{\begin{enumerate}
\item[{\rm (i)}] The set of vertices is the set of all equivalence classes $\widehat{I}_\g \seteq (I_0 \times \ko^\times)/ \sim$,
where the equivalence relation is given in the following way:
$$ (i,x)\sim(j,y) \quad \text{if and only if} \quad V(\varpi_i)_{x} \simeq V(\varpi_j)_{y}.$$
\item[{\rm (ii)}] We put $d$ many arrows from $(i,x)$ to $(j,y)$, where
$d$ is the order of the zero of $d_{i,j}(z)$ at $z=y/x$.
\end{enumerate}
} \label{eq: Quiver S}
\end{eqnarray}

By \cite[(1.7)]{AK}, for $x,y\in \ko^\times$ and $i\in I_0$, and $\g=A^{(t)}_N$
($N\ge 2$, $t=1,2$) or  $\g=D^{(t)}_N$ ($N\ge 4$, $t=1,2$) we have
\eq \label{eq: twi A D iso}
& V(\varpi_i)_x \simeq V(\varpi_i)_{y}\Longleftrightarrow
\begin{cases}
x^2=y^2& \text{if $\g=A^{(2)}_{2n-1}$ ($n\ge2$), $i=n$}\\
x^2=y^2& \text{if $\g=D^{(2)}_{N}$ ($N\ge3$), $1\le i\le N-2$,}\\
x=y&\text{otherwise.}
\end{cases}
\eneq

Let $\Se_0(\g)$ be a connected component of $\Se(\g)$. Then $\Se_0(\g)$ is uniquely determined up to a spectral parameter and hence $\Se_0(\g)$ is unique
up to a quiver isomorphism. For the rest of this paper,
we take $\Se_0(\g)$ as follows:
\begin{equation}
\begin{aligned}
&\Se_0(A^{(1)}_{n}) \seteq \{(i,(-q)^{p}) \in \{1,\ldots,n\} \times \ko^{\times} \ | \ p \equiv d(1,i) \modt \}, \allowdisplaybreaks  \\
&\Se_0(D^{(1)}_{n+1}) \seteq \{( i,(-q)^{p}) \in \{1,\ldots,n,n+1\} \times \ko^{\times} \ | \ p \equiv d(n,i) \modt \}, \allowdisplaybreaks  \\
&\Se_0(A^{(2)}_{2n}) \seteq \{(i,(-q)^{p}) \in \{1,\ldots,n\} \times \ko^{\times}
\ | \ p \in \Z \},  \allowdisplaybreaks  \\
&\Se_0(A^{(2)}_{2n-1}) \seteq \{( i,\pm(-q)^{p})\in \{1,\ldots,n\} \times \ko^{\times}  | \ 1 \le i \le n, \ p \equiv i+1 \modt \}, \allowdisplaybreaks  \\
&\Se_0(D^{(2)}_{n+1}) \seteq \{(i,(\sqrt{-1}) q^{p}),(i', q^{p'}), (n,\pm q^{r}) \in \{1,\ldots,n\} \times \ko^{\times} \ | \  1 \le i,i' \le n-1,  \allowdisplaybreaks  \\
&\hspace{12ex} \ p \equiv n-i \equiv 0 \modt, \ p' \equiv n-i' \equiv 1 \modt, \ r \equiv 0  \modt \}.
\end{aligned}
\end{equation}
Recall that $d(i,j)$ denotes the distance of $i$ and $j$ in the
Dynkin diagram of $\g_0$.

We define the maps
$$\pi_{A^{(1)}_{N}}\col\Se(A^{(1)}_{N}) \to \Se(A^{(2)}_{N})
\quad\text{and}\quad \pi_{D^{(1)}_{N}}\col\Se(D^{(1)}_{N}) \to \Se(D^{(2)}_{N})$$ by
\begin{equation} \label{eq: star A}
\pi_{A^{(1)}_{N}}(i, x)
\seteq
\begin{cases}(i,x)  & \text{ if } 1 \le i \le \big\lfloor \dfrac{N+1}{2} \big\rfloor, \\
( N+1-i,(-1)^{N}  x)   & \text{ if } \big\lfloor \dfrac{N+1}{2} \big\rfloor < i \le N \end{cases}
\end{equation}
and
\begin{equation} \label{eq: star D}
\pi_{D^{(1)}_{N}}(i, x)\seteq
\begin{cases}  \bl i,\,(\sqrt{-1})^{\delta(N \equiv i \modt)+1}\, x\br
\sim \bl i, (\sqrt{-1})^{ N-i}x\br& \text{ if } 1 \le i \le N-2, \\
\bl N-1,(-1)^i x\br   & \text{ if $i \in \{ N-1,N\}$,} \end{cases}
\end{equation}
for $i \in I_0$ and $ x  \in \ko^\times$.

\begin{proposition} \label{prop: quiver iso}\hfill
\bnum
\item The maps $\pi_{A^{(1)}_{N}}$ and $\pi_{D^{(1)}_{N}}$ are
$2:1$ quiver morphisms.
\item
They induce quiver isomorphisms
$$
\Se_0(A^{(1)}_{N}) \overset{\sim}{\Lto} \Se_0(A^{(2)}_{N}) \quad \text{and} \quad
\Se_0(D^{(1)}_{N}) \overset{\sim}{\Lto} \Se_0(D^{(2)}_{N}).
$$
\item  The maps $\pi_{A^{(1)}_{N}}$ and $\pi_{D^{(1)}_{N}}$
commutes with the duality $*$. More precisely, setting $\g^{(t)}=A^{(t)}_N$ or $D^{(t)}_N$,
and denoting by $V^{(t)}(i,x)$ the $\U[\g^{(t)}]$-module $V(\vpi_i)_x$,
$ \bl V\uo(i,x)\br^*\simeq V\uo(j,y)$
implies $ \bl V\ut(\pi_{\g\uo}(i,x))\br^*\simeq V\ut(\pi_{\g\uo}(j,y))$.
\ee
\end{proposition}
\begin{proof}
Our assertions follow from Theorem \ref{thm: denominator of
untwisted}, Theorem \ref{thm: denominator of twisted},
isomorphisms \eqref{eq: twi A D iso} and \eqref{eq:dual1}, \eqref{eq:dual2}.
\end{proof}

By Proposition \ref{prop: quiver iso}, we have the following corollary:
\begin{corollary} \label{cor: FQ2}
Let $\g^{(t)}$ be the affine Kac-Moody algebra of type $A^{(t)}_{N}$ $(N\ge2)$
or  $D^{(2)}_{N}$ $(N\ge4)$ and $\g^{(t)}_0$ its classical part $(t=1,2)$.
For any quiver $Q$ of type $\g\uo_0$
take the  index set
$$ J \seteq  \phi^{-1}(\Pi_{\g\uo_0} \times \{0\})$$
and define the two maps $s\ut\col  J \to I_0\ut $,
 and $X\ut\col J \to \ko^{\times}$ by
$$\bl s\ut(i,p),X\ut(i,p)\br =\pi_{\g\uo}(i,(-q)^p),$$
 where $I_0\ut $ denotes the index set of simple roots for
$\g_0^{(2)}$. Then, the quiver  $\Se^J$ coincides with the
quiver $Q^{\rm{rev}}$ obtained from $Q$ by reversing all the arrows.
\end{corollary}

By Corollary \ref{cor: FQ2}, we obtain an exact KLR-type
duality functor
\begin{equation} \label{eq: FQ2}
\Fun_Q^{(2)}\col \Rep(R) \to \Ca_{\g\ut} \quad \text{ for $\g\ut=A^{(2)}_{N}$
or $D^{(2)}_{N}$.}
\end{equation}
Here $R$ denotes the quiver Hecke algebras of type $\g_0\uo=A_N$ or $D_N$.

\medskip

\subsection{Similarities}
Now we shall observe interesting similarities between the representations of quantum affine algebras of untwisted types and
twisted types thorough the maps $\pi_{A^{(1)}_{N}}$ and $\pi_{D^{(1)}_{N}}$.

\medskip
In the sequel, $\g^{(t)}$ denotes $A^{(t)}_N$ ($N\ge2$) or $D^{(t)}_N$ ($N\ge4$)
for $t=1,2$. We denote by $I\utt$ and $I_0\utt$ the index set of simple roots
for $\g\utt$ and its classical part $\g_0\utt$, respectively.
We denote by $V^{(t)}(\varpi_i)$ the fundamental module for $\U[\g^{(t)}]$.
For $(i,x)\in I_0\utt\times\cor^\times$, we set
$V^{(t)}(i,x)\seteq V\utt(\vpi_i)_x$.

For any Dynkin quiver $Q$ of type $\g\uo_0$ and a height functions $\xi$, we set
$$ \text{
$V^{(1)}_Q(\beta) \seteq V^{(1)}(i,x)$ and
$V_Q\ut(\beta)\seteq V\ut\bl\pi_{\g\uo}(i,x)\br$}$$
where $\phi^{-1}(\beta,0)=(i,p)$ and $x=(-q)^p$.
Here $\phi$ is the map in \eqref{eq: phi}, and $\beta$ is a positive
root in $\Delta^+_{\g\uo}$.

\begin{proposition}[\cite{H01}] \label{prop: same dim}
Let $\g\uo=A^{(1)}_N$ or $D^{(1)}_N$.
For any $(i,x) \in \Se(\g\uo)$, we have $$\dim_{\ko}V^{(1)}(i,x) = \dim_{\ko}V^{(2)}(\pi_{\g}(i,x)).$$
\end{proposition}

\begin{proposition} [{\cite[Theorem 6.1, Theorem 7.1]{CP96}}] \label{prop: Dorey untwisted} \hfill
\begin{enumerate}
\item[{\rm (a)}] Let $\g\uo=A^{(1)}_{N}$ and $(i,x),(j,y),(k,z) \in I_0\uo \times \ko^\times$. Then
$$ \Hom_{U_q'(\g\uo)}\big( V\uo(i,x)\tens V\uo(j,y), V\uo(k,z) \big) \ne 0 $$
if and only if one of the following conditions holds:
\begin{eqnarray}&&
\hs{2ex}\left\{\hs{-1ex}
\ba{l}
\text{{\rm (A-i)}\  $i+j < N+1$, $k=i+j$, $x/z=(-q)^{-j}$, $y/z=(-q)^i$,}
\\[1.5ex]
\text{{\rm (A-ii)}\ $i+j > N+1$, $k=i+j-N-1$, $x/z=(-q)^{-N-1+j}$,
$y/z=(-q)^{N+1-i}$.}\ea
\right.\label{eq: Dorey A untwisted}
\end{eqnarray}
\item[{\rm (b)}]  Let $\g\uo=D^{(1)}_{N}$ and $(i,x),(j,y),(k,z) \in I_0\uo \times \ko^\times$.  Then
$$ \Hom_{U_q'(\g\uo)}\big( V\uo(i,x)\tens V\uo(j,y), V\uo(k,z) \big) \ne 0 $$
if and only if one of the following conditions holds:
\begin{eqnarray}&&
\left\{\hs{2ex}\parbox{73ex}{
\begin{enumerate}
\item[{\rm (D-i)}] $\ell \seteq \max(i,j,k) \le N-2$, $s+m =\ell$ for $\{ s,m \} \seteq \{ i,j,k\} \setminus \{ \ell \}$
and
$$ \big( x/z,y/z \big) =
\begin{cases}
\big( (-q)^{-j},(-q)^{i} \big), & \text{ if } \ell = k,\\
\big( (-q)^{-j},(-q)^{-i+2N-2} \big), & \text{ if } \ell = i,\\
\big( (-q)^{j-2N+2},(-q)^{i} \big), & \text{ if } \ell = j,
\end{cases}
$$
\item[{\rm (D-ii)}] $i+j \geq N$, $k=2N-2-i-j$, $\max(i,j,k) \le N-2$, and $x/z=(-q)^{-j}$, $y/z=(-q)^i$,
\item[{\rm (D-iii)}] $s \seteq \min(i,j,k) \le N-2$, $ \{ m,\ell \}  \seteq \{ i,j,k\} \setminus \{ s \} \subset \{ N-1,N\}$ and
\begin{itemize}
\item $N-s  \equiv  \begin{cases} \ell-m \modt & \text{ if } s = k, \\
\ell-m^*  \modt & \text{ if } s \ne k, \end{cases}$
\item
$\big( x/z,y/z \big) =
\begin{cases}
\big( (-q)^{-N+k+1},(-q)^{N-k-1} \big), & \text{ if } s= k,\\
\big( (-q)^{-N+i+1},(-q)^{2i} \big), & \text{ if } s = i,\\
\big( (-q)^{-2j},(-q)^{N-j-1} \big), & \text{ if } s = j.
\end{cases}
$
\end{itemize}
\end{enumerate}
}\right. \label{eq: Dorey D untwisted}
\end{eqnarray}
\end{enumerate}
\end{proposition}

\begin{remark} \label{rem: Dorey rule for untwisted} \hfill
\begin{enumerate}
\item[{\rm (a)}]
Note that the triples $( (i,x),(j,y),(k,z) )$
in Proposition \ref{prop: Dorey untwisted}
belong to the same connected component of
$\Se(\g\uo)$. $\g\uo=A^{(1)}_{N}$ or $D^{(1)}_{N}$.
\item[{\rm (b)}] Let  $((i,x),(j,y),(k,z))$ be a triple appearing in Proposition \ref{prop: Dorey untwisted}. Then
the normalized R-matrix
$\Rnorm_{V\uo(\varpi_i),V\uo(\varpi_j)}(z)$ has a multiple pole at $z=y/x$
if and only if $((i,x),(j,y),(k,z))$ satisfies the condition {\rm
(D-ii)} in \eqref{eq: Dorey D untwisted}. (See
Lemma 3.2.4 in \cite{KKK13B}.)
\end{enumerate}
\end{remark}

\begin{proposition}[{\cite[Section 3]{Oh14}, \cite[Corollary 3.3.2]{KKKO14A}}] \label{prop: Dorey twisted} \hfill
\begin{enumerate}
\item[{\rm (a)}] Let $(i,x),(j,y),(k,z) \in \Se_0(A^{(2)}_N)$.
Assume that $$\text{$(i,x) = \pi_{A^{(1)}_N} (i',x')$, $(j,y) = \pi_{A^{(1)}_N} (j',y')$ and
$(k,z) = \pi_{A^{(1)}_N} (k',z')$}$$ such that $(i',x'),(j',y'),(k',z') \in \Se_0(A^{(1)}_N)$ satisfy one of the conditions
in \eqref{eq: Dorey A untwisted}.
Then we have an epimorphism
$$ V^{(2)}(i,x)\tens V^{(2)}(j,y)\epito V^{(2)}(k,z). $$
\item[{\rm (b)}] Let $(i,x),(j,y),(k,z) \in \Se_0(D^{(2)}_N)$.
Assume that $$\text{$(i,x) = \pi_{D^{(1)}_N} (i',x')$, $(j,y) =
\pi_{D^{(1)}_N} (j',y')$ and $(k,z) = \pi_{D^{(1)}_N} (k',z')$}$$
such that $(i',x'),(j',y'),(k',z') \in \Se_0(D^{(1)}_N)$ satisfy one
of the conditions {\rm (D-i)} or  {\rm (D-iii)} in
\eqref{eq: Dorey D untwisted}
  Then we have an epimorphism
$$ V^{(2)}(i,x)\tens V^{(2)}(j,y)\epito V^{(2)}(k,z). $$
\end{enumerate}
\end{proposition}

\begin{proposition} [\cite{Oh14A,Oh14D}] \label{prop: minimal liked}
Let $\g\uo=A^{(1)}_{N}$ \ro respectively,  $D^{(1)}_{N}$\rf, and let $Q$ be a
Dynkin quiver of type $\g_0\uo$.
 Assume that $(\beta,\gamma)$ is a minimal pair of $\al\in\Delta^+_{\g_0\uo}$
with respect to the convex total order induced by a reduced expression
$\redex$ of $w_0$ adapted to $Q$,
and  write
$$\text{$\phi^{-1}(\beta,0)=(i,a)$, $\phi^{-1}(\gamma,0)=(j,b)$ and
$\phi^{-1}(\al,0)=(k,c)$.}$$
Then the triple $(i,(-q)^a),(j,(-q)^b),(k,(-q)^c)$ satisfies one of
the conditions {\rm (A-i)} or {\rm (A-ii)} in \eqref{eq:
Dorey A untwisted} {\rm(}respectively,  {\rm (D-i)} or
{\rm (D-iii)} in \eqref{eq: Dorey D untwisted}{\rm)}. In particular,
there is a surjective homomorphism
$$V_Q\utt(\gamma)\tens V_Q\utt(\beta)\epito V_Q^{(t)}(\al).$$
\end{proposition}

In Section \ref{Sec: Dorey's rule}, we will prove the following
theorem by using the exact functor $\Fun^{(2)}_Q$.
Note that Proposition~\ref{prop: Dorey twisted}
is a special case of this theorem.

\begin{theorem} \label{thm: twist Dorey}
Let $\g\uo=A^{(1)}_{N}$ or $D^{(1)}_{N}$, and let $(i,x),(j,y),(k,z)
\in \Se(\g\uo)$. Then there exists an epimorphism
$$V^{(2)}(\pi_{\g\uo}(i,x))\tens V^{(2)}(\pi_{\g\uo}(j,y)) \epito V^{(2)}(\pi_{\g\uo}(k,z))$$ if and only if there exists an epimorphism
$$V^{(1)}(i,x)\tens V^{(1)}(j,y) \epito V^{(1)}(k,z).$$
\end{theorem}
Note that, for an arbitrary quantum affine algebra $\U$,
$V(\vpi_i)_x\tens V(\vpi_j)_y$ has a simple head by
\cite{KKKO14}.
Hence we have
$$\dim \Hom_{\U}\bl V(\vpi_i)_x\tens V(\vpi_j)_y, V(\vpi_k)_z\br\le1.$$

\section{The functor $\QFun{2}$}
In this section, we investigate the properties of the exact functor
$\Fun^{(2)}_Q$ given in \eqref{eq: FQ2}.
Recall that $R$ is the symmetric quiver Hecke algebra of type $\g_0\uo$.

\begin{theorem} \label{thm: FQ2 SQ}
For any Dynkin quiver $Q$ of type $\g_0\uo$
and $\al\in\Delta^+_{\g_0\uo}$, we have
$$\Fun^{(2)}_Q(S_Q(\al)) \simeq V^{(2)}_Q(\al),$$
where $S_Q(\al)$ denotes the $R(\al)$-module defined in {\rm Remark \ref{rem: SQ(beta)}}.
\end{theorem}

\begin{proof}
We shall prove our assertion by  induction on $\Ht(\al)$. For
$\al$ with $\Ht(\al)=1$, our assertion follows from Theorem
\ref{thm: QASWD functor} (a). Now we assume that $\Ht(\al) \ge
2$. Note that there exists a minimal pair $(\beta,\gamma)$ of
$\al$. By Theorem \ref{thm: PBW} (c), we have a six-term exact
sequence of $R(\al)$-modules
$$ 0 \to S_{Q}(\al)
\Lto S_{Q}(\beta) \conv S_{Q}(\gamma) \overset{r}{\Lto}
S_{Q}(\gamma) \conv S_{Q}( \beta ) \overset{s}{\Lto}
S_{Q}(\al) \to 0.$$ Applying the functor $\Fun^{(2)}_Q$,
 we have an exact sequence of $U_q'(\g)$-modules by the induction
hypothesis
$$ 0 \to \Fun^{(2)}_Q(S_{Q}(\al))
\to V^{(2)}_{Q}(\beta) \tens V^{(2)}_{Q}(\gamma) \overset{\Fun^{(2)}_Q(r)}{\Lto}
V^{(2)}_{Q}(\gamma) \tens V^{(2)}_{Q}(\beta) \overset{\Fun^{(2)}_Q(s)}{\Lto} \Fun^{(2)}_Q(S_{Q}(\al)) \to 0.$$
On the other hand, Proposition~\ref{prop: minimal liked} implies that
$V^{(2)}_{Q}(\gamma) \tens V^{(2)}_{Q}(\beta)$ is not simple.

We have then
$\Fun^{(2)}_Q(S_{Q}(\al)) \ne 0$. Indeed, if it vanished, we would have
$$V^{(2)}_{Q}(\beta) \tens V^{(2)}_{Q}(\gamma) \simeq V^{(2)}_{Q}(\gamma)
\tens V^{(2)}_{Q}(\beta),$$
which implies that $V^{(2)}_{Q}(\gamma) \tens V^{(2)}_{Q}(\beta)$ is simple
by \cite[Corollary 3.16]{KKKO14}.

Hence $\Fun^{(2)}_Q(S_{Q}(\al))$ is the image of
a non-zero homomorphism
$V^{(2)}_{Q}(\gamma) \tens V^{(2)}_{Q}(\beta)\to
 V^{(2)}_{Q}(\beta) \tens V^{(2)}_{Q}(\gamma) $.
Hence, \cite{KKKO14} and the quantum affine version of
 \cite[Proposition 4.5]{KKKO15B} imply that $\Fun^{(2)}_Q(S_{Q}(\al))$ is the simple head of
$V^{(2)}_{Q}(\gamma) \tens V^{(2)}_{Q}(\beta)$.

On the other hand,
$V^{(2)}_{Q}(\al)$ is a simple quotient of
$V^{(2)}_{Q}(\gamma) \tens V^{(2)}_{Q}(\beta)$
by Proposition~\ref{prop: minimal liked}.
Hence we obtain the desired result
$\Fun^{(2)}_Q(S_{Q}(\al))\simeq V^{(2)}_{Q}(\al)$.
\end{proof}

\begin{lemma} \label{lem: k<l no poles} Let $\beta,\gamma\in \Delta^+_{\g_0\uo}$
and $t=1,2$.
If $ \Rnorm_{V\utt_Q(\beta),V\utt_Q(\gamma)}(z)$ has a pole at $z=1$, then
$\gamma\prec_Q \beta$.
\end{lemma}
\begin{proof}
By Proposition \ref{prop: quiver iso}, it is enough to prove the case $t=1$.
Set $\phi^{-1}(\beta,0)=(i,a)$ and $\phi^{-1}(\gamma,0)=(j,b)$.
Assume that $ \Rnorm_{V\utt_Q(\beta),V\utt_Q(\gamma)}(z)$ has a pole at $z=1$, or equivalently,
the denominator $d_{ij}(z)$ has a zero at $z=(-q)^{b-a}$.
Then we can easily see
that $d(i,j)+2\le b-a$
by Theorem \ref{thm: denominator of untwisted}.
Here $d(i,j)$ is the distance of $i$ and $j$ in the Dynkin diagram.
Then $\gamma\preceq_Q\beta$ by the definition \eqref{eq: preceq Q}.
\end{proof}

For $\{m_\beta\}_{\beta \in \Delta^+_{\g_0\uo}}
\in \bl\Z_{\ge 0}\br^{\Delta^+_{\g_0\uo}}$, let us denote
by $\mathop{\tens\limits^{\rightarrow}}\nolimits_{\beta \in \Delta^+_{\g_0\uo}}\,(V^{(2)}_Q(\beta))^{ \tens m_\beta} )$  the increasingly ordered tensor product with respect to a total order on $\Delta^+_{\g_0\uo}$ stronger than  $\preceq_Q$.
The isomorphism class of the module
$ \mathop{\tens\limits^{\rightarrow}}\nolimits_{\beta \in \Delta^+_{\g_0\uo}}\,(V^{(2)}_Q(\beta))^{ \tens m_\beta}$ does not depend on the choice of such a total order
by Lemma~\ref{lem: k<l no poles}.
Note also that $\tens\limits^{\rightarrow}\, (V^{(2)}_Q(\beta))^{ \tens m_\beta} )$ has a
simple head by Lemma~\ref{lem: k<l no poles} and
\cite[{Theorem 2.2.1 (ii)}]{KKK13A}.

\begin{definition} \label{def:CQ2}
 The category $\Ca^{(2)}_Q$ is
the full subcategory of $\Ca_{\g\ut}$ consisting of objects
$M$ such that any  simple subquotient of $M$ is isomorphic to
$ {\rm hd}(\mathop{\tens\limits^{\rightarrow}}\nolimits_{\beta \in \Delta^+_{\g_0\uo}}\,(V^{(2)}_Q(\beta))^{ \tens m_\beta} )$ for some $\{m_\beta\}_{\beta \in \Delta^+_{\g_0\uo}}
\in \bl\Z_{\ge 0}\br^{\Delta^+_{\g_0\uo}}$.
\end{definition}
Note that there is a similar characterization of $\Ca_{Q}\uo$.

By the definition, $\Ca_Q\ut$ is stable under taking subquotients
and extensions. Note that we have \eq&& \text{If $V\utt(\vpi_i)_x$
belongs to $\Ca^{(t)}_Q$, then it is isomorphic to $V_Q\utt(\beta)$
for some $\beta\in\Delta_{\g_0\uo}^+$.} \label{eq:fundQ} \eneq
Indeed, the dominant extremal weight of
$\mathrm{hd}\bl\tens\limits^{\rightarrow}\, (V^{(t)}_Q(\beta))^{
\tens m_\beta} )\br$ is equal to $\sum_{\beta \in
\Delta^+_{\g_0\uo}}m_\beta \varpi_{i_\beta}$ for $t=1$
(respectively, is equal to   $\sum_{\beta \in
\Delta^+_{\g_0\uo}}m_\beta  \varpi_{i'_\beta}$ for $t=2$), where
$\phi(i_\beta, p_\beta) = (\beta,0)$ and $i'_\beta$ denotes the
first component of $\pi_{\g\uo}(i_\beta, (-q)^{p_\beta})$ for $\beta
\in \Delta^+_{\g_0\uo}$.

The proof of the following theorem works also for $\Fun_Q^{(1)}$,
and it provides an alternative proof of
\cite[Corollary 4.3.6]{KKK13B}.

\begin{theorem} \label{thm: simples to simples}
The functor $\Fun_Q^{(2)}$ sends a  simple module to a
simple module.
Moreover, the functor  $\Fun^{(2)}_Q$ induces a bijection between
the set of simple modules in $\Rep(R)$ and
the set of simple modules in $\Ca^{(2)}_Q$.
\end{theorem}
\begin{proof}
By Theorem \ref{thm: PBW}, every simple module $M$ in $\Rep(R)$ is isomorphic to the image of the homomorphism
\begin{align} \label{eq: simple step 1}
\rmat{\um} \colon
\overset{\Lto}{\nabla}_{Q}(\um) \seteq \dconv{k=1}{r}S_{Q}(\beta_k)^{\conv m_k}
\to \overset{\Lgets}{\nabla}_{Q}(\um) \seteq \dconv{k=1}{r}S_{Q}(\beta_{r-k+1})^{\conv m_{r-k+1}}
\end{align}
for a unique $\um \in \Z_{\ge 0}^r$.
Moreover we have
\eq
&&[\overset{\Lto}{\nabla}_{Q}(\um)] \in [{\rm Im}(\rmat{\um})] +
\sum_{\um' \prec_\Z \um} \Z_{\ge 0} [{\rm Im}(\rmat{\um'})].
\label{eq:P}
\eneq
Applying the functor $\Fun^{(2)}_Q$ to \eqref{eq: simple step 1}, we obtain
$$  \xymatrix@C=9ex{ \overset{\Lto}{W}_Q(\um) \seteq  \dtens{k=1}{r} V^{(2)}_Q(\beta_k)^{\tens m_k} \ar[r]^{\Fun^{(t)}_Q(\rmat{\um}) \qquad } &
\overset{\Lgets}{W}_Q(\um) \seteq \dtens{k=1}{r}
V^{(2)}_Q(\beta_{r-k+1})^{\tens m_{r-k+1}}}.$$

Now we shall prove that
\eq&&\text{${\rm Im}\big(\Fun^{(2)}_Q(\rmat{\um})\big) \simeq \Fun^{(2)}_Q\big( {\rm Im}(\rmat{\um})\big)$ is simple
and isomorphic to ${\rm hd}(\overset{\Lto}{W}_Q(\um) )$.}\label{cond:simple}
\eneq

If $\um$ is a minimal element of $\Z^r_{\ge 0}\setminus\{0\}$
, i.e. $\um$ is a unit vector, then \eqref{cond:simple}
follows from Theorem~\ref{thm: FQ2 SQ}.
Assume that $\um$ is not minimal.
By Lemma \ref{lem: k<l no poles}, we can apply
\cite[{Theorem 2.2.1 (ii)}]{KKK13A}.
Then,
$\overset{\Lto}{W}_Q(\um)$ has a simple head which is equal to
the image of any non-zero map from $\overset{\Lto}{W}_Q(\um)$ to
$\overset{\Lgets}{W}_Q(\um)$.
Thus it is enough to show that $\Fun^{(t)}_Q(\rmat{\um})$ is non-zero.

By the induction hypothesis on $\prec_\Z$,
every composition
factor of $\Ker\big(\Fun^{(t)}_Q(\rmat{\um})\big)$ is of the form
${\rm hd}(\overset{\Lto}{W}_Q(\um'))$ for some $\um' \prec_\Z \um$.
By \cite[Theorem 2.2.1 (iii)]{KKK13A},   ${\rm
hd}(\overset{\Lto}{W}_Q(\um))$ is  isomorphic to ${\rm
hd}(\overset{\Lto}{W}_Q(\um'))$ if and only if $\um' = \um$.
Thus we conclude that ${\rm hd}(\overset{\Lto}{W}_Q(\um))$
cannot appear as a composition factor of
$\Ker\big(\Fun^{(t)}_Q(\rmat{\um})\big)$, which yields that
$\Fun^{(t)}_Q(\rmat{\um})$ is non-zero.
\end{proof}

\begin{corollary}\label{cor:simple to simple} \hfill
\begin{enumerate}
\item[{\rm (i)}]The functor $\Fun^{(2)}_Q$ induces a functor
$\Rep(R)\to \Ca^{(2)}_Q$.
\item[{\rm (ii)}] The functor $\Fun^{(2)}_Q$ is faithful.
\item[{\rm (iii)}]
The category $\Ca^{(2)}_Q$ is stable under taking tensor products.
\end{enumerate}
\end{corollary}
\begin{proof}
\rm (i) follows immediately from Theorem~\ref{thm: simples to simples}.

\noi
{\rm (ii)} Since $\Fun^{(2)}_Q$ sends simples to simples,
the module ${\rm Im}({\Fun^{(2)}_Q}(f))$ does not vanish for any non-zero homomorphism $f$.
Hence ${\Fun^{(2)}_Q}(f)$ is non-zero.

\smallskip\noi
{\rm (iii)} follows from the fact that $\Fun^{(2)}_Q\col \Rep(R)\to\Ca_Q^{(2)}$
sends a convolution product to a tensor product.
\end{proof}

By this corollary, we have another characterization of $\Ca^{(2)}_Q$:
$\Ca^{(2)}_Q$ is the smallest full subcategory
of $\Ca_{\g\ut}$ such that it is stable under taking subquotients, extensions,
tensor products and it contains $V_Q\ut(\al_i)$ for all $i\in I_0\uo$.

\begin{corollary} \label{cor: FQ1 FQ2}
The functor $\Fun^{(2)}_Q$ induces a ring isomorphism
$$\varphi^{(2)}:  \K(\Rep(R))\isoto \K(\Ca^{(2)}_Q)$$ and hence a ring isomorphism
$$\varphi \seteq \varphi^{(2)} \circ (\varphi^{(1)})^{-1}: \K(\Ca^{(1)}_Q) \isoto \K(\Ca^{(2)}_Q).$$
Moreover, the ring isomorphism $\varphi$ sends simples to simples and preserves the dimensions.
\end{corollary}

\begin{proof}

The first assertion follows from the preceding corollary and the fact that $\Fun_Q^{(2)}$ sends a convolution product to a tensor product. For the last assertion on the dimension, we can apply the same induction argument of \cite[Theorem 3.5.10]{KKKO14A} along with Proposition \ref{prop: same dim}.
\end{proof}

\begin{conjecture}
The functor ${\mathcal F}_{Q}^{(t)}\col\Rep(R)\to\Ca_Q^{(t)}$ $(t=1,2)$ is
an equivalence of categories.
\end{conjecture}

\vskip 3mm

\section{Proof of Theorem \ref{thm: twist Dorey}} \label{Sec: Dorey's rule}
In this section, we give a proof of Theorem \ref{thm: twist
Dorey} for each $\g$, which we have postponed in Section \ref{sec:
similar}.

\begin{lemma}\label{lem:3}
Let $\g\uo=A^{(1)}_N$ or $D^{(1)}_N$. Assume that  $(i,x)$, $(j,y)$ is
a pair of elements of $\Se(\g\uo)$ such that they are
connected by an arrow and that
$V\uo(\vpi_i)_x$ and $V\uo(\vpi_j)_y$ are not dual to each other. Then there
exist a Dynkin quiver of type $\g_0\uo$,
a height function $\xi$, $a\in\ko^\times$ and
$(i,s)$, $(j,t)\in \Gamma_{Q}$ such that
$((i,x),(j,y))=((i,a(-q)^s),(j,a(-q)^t))$.
\end{lemma}

Postponing the proof of this lemma, we shall prove
Theorem~\ref{thm: twist Dorey}.

\begin{proof}[Proof of Theorem~\ref{thm: twist Dorey}]
Assume that there is an epimorphism
\eq
&& V^{(1)}(i,x)\tens V^{(1)}(j,y) \epito V^{(1)}(k,z)\label{eq:epi1}
\eneq
 or an epimorphism
\eq
&&V^{(2)}(\pi_{\g\uo}(i,x))\tens V^{(2)}(\pi_{\g\uo}(j,y))
\epito V^{(2)}(\pi_{\g\uo}(k,z)).\label{eq:epi2}\eneq
Then $(i,x)$ and $(j,y)\in \Se(\g\uo)$ are connected by an arrow.
Moreover $V^{(1)}(i,x)$ and $V^{(1)}(j,y)$ are not dual to each other.
Indeed, if $V^{(1)}(i,x)$ and $V^{(1)}(j,y)$ were dual to each other, then
$V^{(2)}(\pi_{\g\uo}(i,x))$ and $V^{(2)}(\pi_{\g\uo}(j,y))$ would be also dual to each other by Proposition~\ref{prop: quiver iso} (iii). Then,
there would not exist an epimorphism \eqref{eq:epi1} nor \eqref{eq:epi2}.

Hence, by Lemma~\ref{lem:3}, there exist a Dynkin quiver $Q$ of type
$\g_0\uo$, a height function $\xi$, $a\in\ko^\times$ and $(i,s)$,
$(j,t)\in \Gamma_{ Q}$ such that
$((i,x),(j,y))=((i,a(-q)^s),(j,a(-q)^t))$. We may assume $a=1$
without loss of generality. Then, $V^{(1)}(i,x)$ and $V^{(1)}(j,y)$
belong to $\Ca_Q\uo$,  and also $V^{(2)}(\pi_{\g\uo}(i,x))$ and
$V^{(2)}(\pi_{\g\uo}(j,y))$ belong to $\Ca_Q\ut$. Hence $
V^{(1)}(k,z)$ belongs to $\Ca_Q\uo$ or $V^{(2)}(\pi_{\g\uo}
(k,z)) $ belongs to
 $\Ca_Q\ut$. By \eqref{eq:fundQ}, there are $\al,\beta,\gamma\in\Delta_{\g_0\uo}^+$ such that
$ V^{(1)}(i,x)=V_Q\uo(\al)$, $V^{(1)}(j,y)=V_Q\uo(\beta)$,
$V^{(1)}(k,z)=V_Q\uo(\gamma)$, $V^{(2)}(\pi_{\g\uo}(i,x))=V_Q\ut(\al)$,
$V^{(2)}(\pi_{\g\uo}(j,y))= V_Q\ut(\beta)$ and
$V^{(2)}(\pi_{\g\uo}(k,z))=V_Q\ut(\gamma)$.
Let $f\col S_Q(\al)\conv S_Q(\beta)\to S_Q(\beta)\conv S_Q(\al)$ be a
non-zero morphism in $\Rep(R)$.
Then $\Im(f)$ is the simple head of $S_Q(\al)\conv S_Q(\beta)$
and $\Im\bl\Fun\utt_Q(f)\br$ is the simple head of
$V_Q^{(t)}(\al)\tens V_Q^{(t)}(\beta)$.
Therefore, for each $t=1,2$, we have the equivalence relations:
\eqn
&&\text{there exists an epimorphism
$V_Q^{(t)}(\al)\tens V_Q^{(t)}(\beta)\epito  V_Q^{(t)}(\gamma)$}\\
&&\Longleftrightarrow
\Im\bl\Fun\utt_Q(f)\br\simeq V_Q^{(t)}(\gamma)\simeq\Fun\utt\bl S_Q(\gamma)\br\\
&&\Longleftrightarrow\Im(f)\simeq S_Q(\gamma).
\eneqn
Thus we obtain Theorem~\ref{thm: twist Dorey}.
\end{proof}

\medskip
Before starting the proof of
Lemma~\ref{lem:3}, let us remark on  the Auslander-Reiten quiver $\Gamma_Q$.

The set of vertices $(\Gamma_Q)_0$ of $\Gamma_Q$  is described as follows:
\eq
&&
(\Gamma_Q)_0 =\{ (i,p)\in I_0\times\Z \mid \xi_i -2m_i \le p \le \xi_i,
\ p\equiv\xi_i\modt \}
\label{eq: mi}\\
&&\hs{30ex}\text{where $m_i \seteq \max \{ k \in \Z_{\ge 0}\mid
\tau^k(\gamma_i) \in \Delta^+_{\g_0} \}$.}\nn
\eneq

\begin{example} \label{ex: quiver Q A} \hfill
\begin{itemize}
\item[{\rm (a)}] For a quiver $\overset{\to}{Q}\seteq \xymatrix@R=0.5ex{ *{ \bullet }<3pt> \ar@{->}[r]_<{1}  &*{\bullet}<3pt>
\ar@{->}[r]_<{2}  &*{ \bullet }<3pt>
\ar@{.>}[rr] &&*{\bullet}<3pt>
\ar@{->}[r]_<{N-1}  & *{\bullet}<3pt> \ar@{-}[l]^<{\ \ N}
}$ of type $A_{N}$, we have $m_i= N-i$ for all $i \in I_0$.
\item[({\rm b})] For a quiver $\overset{\gets}{Q}\seteq \xymatrix@R=0.5ex{ *{ \bullet }<3pt> \ar@{<-}[r]_<{1}  &*{\bullet}<3pt>
\ar@{<-}[r]_<{2}  &*{ \bullet }<3pt> \ar@{<.}[rr] &&*{\bullet}<3pt>
\ar@{<-}[r]_<{N-1}  & *{\bullet}<3pt> \ar@{-}[l]^<{\ \ N} }$ of type
$A_{N}$, we have $m_i= i-1$ for all $i \in I_0$.
\end{itemize}
\end{example}
Note that, for any Dynkin quiver $Q$ of type $D_N$ and any $1 \le i \le N$, $m_i$'s in \eqref{eq: mi} are given as follows (\cite[Lemma 1.11]{Oh14D}):
\begin{equation} \label{eq: D m_i}
\begin{cases}
m_i = N-2 &  \text{ if } 1 \le i \le N-2, \\
m_{N-1}=N-3, \ m_N=N-1 & \text{ if } N \equiv 1 \modt \text{ and } \xi_{N}=\xi_{N-1}+2, \\
m_{N-1}=N-1, \ m_N=N-3 & \text{ if } N \equiv 1 \modt \text{ and }\xi_{N-1}=\xi_{N}+2,\\
m_{N-1}=m_N=N-2 & \text{ otherwise. }
\end{cases}
\end{equation}

\medskip
\begin{proof}[Proof of Lemma~\ref{lem:3}]
Without loss of generality, we may assume $(i,x)$, $(j,y)\in\Se_0(\g\uo)$.
Set $(i,x)=(i,(-q)^s)$ and $(j,y)=(j,(-q)^t)$. By the
assumption, $(-q)^{\ell}$ is a root of $d_{i,j}(z)$ for $\ell=|s-t|>0$.
In this proof, we only consider the case when $i \ge j$ and $s < t$. The other cases can be proved by similar arguments.

\noindent
(Case $\g=A^{(1)}_{N}$) \quad
Note that
the case $i+j=t-s=N+1$ is excluded by the assumption.
By \eqref{eq: denominator A1}, we have
$$\text{$i-j+2\le t-s\le \min(i+j, 2(N+1)-i-j)$ and $t-s\equiv i-j\modt$.}$$
Hence, one of the following two conditions holds:
\bnum
\item $t-s\le i+j-2$,
\item $t-s\le 2N-i-j$.
\ee
Indeed, otherwise, we have
$i+j, \ 2N-i-j+2\le t-s\le \min(i+j, 2(N+1)-i-j)$, which implies
$i+j=t-s=N+1$. It is a contradiction.

\medskip
\noi
In case (i), we set $\xi_k=t+k-j$. Then we have $m_k=k-1$ by Example~\ref{ex: quiver Q A} (b), and
$t-i-j+2=\xi_i-2m_i\le s\le \xi_i=t+i-j$.

\smallskip
\noi
In case (ii), we set $\xi_k=t-k+j$. Then we have $m_k=N-k$ by Example~\ref{ex: quiver Q A} (a), and
$t-2N+i+j=\xi_i-2m_i\le s\le \xi_i=t-i+j$.

\bigskip
\noindent
Now let us consider the case $\g=D^{(1)}_{N}$.

\smallskip
\noi
(Case $\g=D^{(1)}_{N}$, $1\le j\le N-2$)\quad
In this case we may assume that
$i\le N-1$. Choosing $\xi_{N-1}=\xi_N$, we may assume $m_k=N-2$.
By \eqref{eq: denominator D1}, we have
$i-j+2\le t-s\le 2N-2+j-i$ and $t-s\equiv i-j\modt$.
Hence, by taking $\xi_j=t$, and $\xi_i=t-i+j$,
we have $\xi_i-2m_i=t-i+j-2N+4\le s\le \xi_i$ except the case
$i=j$ and $t-s=2N-2$. The last case is excluded by the assumption that
$V(\vpi_i)_x$ and $V(\vpi_j)_y$ are not dual to each other.

\medskip
\noi
(Case $\g=D^{(1)}_{N}$, $i,j\in\{N-1,N\}$)\quad
By \eqref{eq: denominator D1}, we have
$2\le t-s\le 2N-2$ and $t\equiv s\modt$.

Assume first that $N$ is even. Then the case $i-j=t-s-2(N-1)=0$ is excluded by the assumption.
By \eqref{eq: D m_i}, we have $m_k=N-2$.
Hence it is enough to take $\xi_j=t$ and
$\xi_i=t-2\delta(i\not=j)$. Then $\xi_i-2m_i\le s\le \xi_i$.

Assume next that $N$ is odd.
Then the case $i\not=j$, $t-s=2(N-1)$ is excluded.
If $i=j$ then we take $m_i=N-1$, $\xi_i=t$.
If $i\not=j$, then $0\le t-s\le 2(N-2)$ and we take
$\xi_i=\xi_j=t$, $m_i=N-2$. Then $\xi_i-2m_i\le s\le \xi_i$.
\end{proof}


\bibliographystyle{amsplain}


\end{document}